\documentclass{article}

\PassOptionsToPackage{numbers, compress}{natbib}
\usepackage[preprint]{neurips_2023}

\usepackage[utf8]{inputenc} 
\usepackage[T1]{fontenc}    
\usepackage{hyperref}       
\usepackage{url}            
\usepackage{booktabs}       
\usepackage{amsfonts}       
\usepackage{nicefrac}       
\usepackage{microtype}      
\usepackage[table,xcdraw]{xcolor}        

\usepackage{mathtools} 
\usepackage{graphicx} 
\usepackage{amsmath,amsthm,amssymb}

 \usepackage[table,xcdraw]{xcolor}
 \usepackage{enumitem}
 \usepackage{mdframed}
 \usepackage{dsfont}
 
 \usepackage{nccmath}
 
 \setcitestyle{square}

\def\1{1\!{\rm l}}

\definecolor{blendedblue}{rgb}{0.2,0.2,0.7}

\DeclareMathAlphabet{\mathpzc}{OT1}{pzc}{m}{it}

\newcommand{\N}{\mathbb{N}}

\newcommand{\given}{\,|\,}

\newcommand{\R}{\mathds{R}}

\newcommand{\bi}{\begin{enumerate}[label=\roman*)]}
\newcommand{\ei}{\end{enumerate}}
\newcommand{\ba}{\begin{array}{rcl}}
\newcommand{\ea}{\end{array}}

\makeatletter
\newcommand*\bigcdot{\mathpalette\bigcdot@{.5}}
\newcommand*\bigcdot@[2]{\mathbin{\vcenter{\hbox{\scalebox{#2}{$\m@th#1\bullet$}}}}}
\makeatother

\newcommand\restr[2]{{
  \left.\kern-\nulldelimiterspace 
  #1 
  \vphantom{\big|} 
  \right|_{#2} 
  }}

\newcommand{\norm}[1]{\left\lVert#1\right\rVert}

\newtheorem{proposition}{Proposition}
\newtheorem{corollary}{Corollary}
\newtheorem{theorem}{Theorem}
\newtheorem{lemma}{Lemma}
\newtheorem{remark}{Remark}
\newtheorem{definition}{Definition}

\title{Variational Gaussian Processes For Linear Inverse Problems}

\author{%
  Thibault Randrianarisoa \\
  Department of Decision Sciences\\
  Bocconi University\\
  via Roentgen 1, 20136, Milano, MI, Italy \\
  \texttt{thibault.randrianarisoa@unibocconi.it} \\
   \And
   Botond Szabo \\
   Department of Decision Sciences \\
   Bocconi University \\
   via Roentgen 1, 20136, Milano, MI, Italy\\
   \texttt{botond.szabo@unibocconi.it} \\
}

\begin{document}

\maketitle

\begin{abstract}
By now Bayesian methods are routinely used in practice for solving inverse problems. In inverse problems the parameter or signal of interest is observed only indirectly, as an image of a given map, and the observations are typically further corrupted with noise. Bayes offers a natural way to regularize these problems via the prior distribution and provides a probabilistic solution, quantifying the remaining uncertainty in the problem. However, the computational costs of standard, sampling based Bayesian approaches can be overly large in such complex models. Therefore, in practice variational Bayes is becoming increasingly popular. Nevertheless, the theoretical understanding of these methods is still relatively limited, especially in context of inverse problems. In our analysis we investigate variational Bayesian methods for Gaussian process priors to solve linear inverse problems. We consider both mildly and severely ill-posed inverse problems and work with the popular inducing variables variational Bayes approach proposed by Titsias \cite{pmlr-v5-titsias09a}. We derive posterior contraction rates for the variational posterior in general settings and show that the minimax estimation rate can be attained by correctly tunned procedures. As specific examples we consider a collection of inverse problems including the heat equation, Volterra operator and Radon transform and inducing variable methods based on population and empirical spectral features.
\end{abstract}


\section{Introduction}\label{sec: intro}

In inverse problems we only observe the object of interest (i.e. function or signal) indirectly, through a transformation with respect to some given operator. Furthermore, the data is typically corrupted with measurement error or noise.  In practice the inverse problems are often ill-posed, i.e. the inverse of the operator is not continuous. Based on the level of ill-posedness we distinguish mildly and severely ill-posed cases. The ill-posedness of the problem prevents us from simply inverting the operator as it would blow up the measurement errors in the model. Therefore, to overcome this problem, regularization techniques are applied by introducing a penalty term in the maximum likelihood approximation. Standard examples include  generalized Tikhonov, total variation and Moore-Penrose estimators, see for instance \cite{MR885511,MR2361904,candes2008introduction,MR2421941,MR2047686,tihonov1963solution} or a recent survey \cite{arridge2019solving} on data-driven methods for solving inverse problems

An increasingly popular approach to introduce regularity to the model is via the Bayesian paradigm, see for instance \cite{arridge2019solving,briol2019probabilistic, MR2608372,law2015data,stuart2010inverse} and references therein. Beside regularization, Bayesian methods provide a probabilistic solution to the problem, which can be directly used to quantify the remaining uncertainty of the approach. This is visualised by plotting credible sets, which are sets accumulating prescribed percentage of the posterior mass. For computing the posterior typically MCMC algorithms are used, however, these can scale poorly with increasing sample size due to the complex structure of the likelihood. Therefore, in practice often alternative, approximation methods are used. Variational Bayes (VB) casts the approximation of the posterior into an optimization problem. The VB approach became increasingly popular to scale up Bayesian inverse problems, see for instance the recent papers \cite{hegde2022variational,maestrini2021variational,meng2022sparse,povala2022variational} and references therein. However, until recently these procedures were considered black box methods basically without any theoretical underpinning. Theoretical results  are just starting to emerge \cite{MR4124331, ray2022variational,MR4011769, MR4102680, zhang:gao:vb}, but we still have limited understanding of these procedures in complex models, like inverse problems.

In our analysis we consider Gaussian process (GP) priors for solving linear inverse problems. For Gaussian likelihoods, due to conjugacy, the corresponding posterior has an analytic form. Nevertheless, they are applied more widely, in non-conjugate settings as well. However, training and prediction even in the standard GP regression model, scales as $O\left(n^3\right)$ (or $O\left(n^2\right)$ for exact inference in the recent paper \cite{GPyTorch} leveraging advances in computing hardware) and $O\left(n^2\right)$, respectively, which practically limits GPs to a sample size $n$ of order $10^4$. Therefore, in practice often not the full posterior, but an approximation is computed. Various such approximation methods were proposed based on some sparse or low rank structure, see for instance \cite{CsatoThesis,Csato2002,Matthews2016,10.5555/1046920.1194909,Seeger2003b,Seeger2003a,pmlr-vR4-seeger03a,NIPS2005_4491777b,pmlr-v5-titsias09a}. Our focus here lies on the increasingly popular inducing variable variational Bayes method introduced in \cite{pmlr-v5-titsias09a,titsias2009variational}.

In our work we extend the inducing variable method for linear inverse problems and derive theoretical guarantees for the corresponding variational approximations. More concretely we adopt a frequentist Bayes point-of-view in our analysis by assuming that there exists a true data generating functional parameter of interest and investigate how well the variational posterior can recover this object. We derive contraction rates for the VB posterior around the true function both in the mildly and severely ill-posed inverse problems. We then focus on two specific inducing variable methods based on the spectral features of the prior covariance kernel. We show that for both methods if the number of inducing variables are chosen large enough for appropriately tunned priors the corresponding variational posterior concentrates around the true function with the optimal minimax estimation rate. One, perhaps surprising aspect of the derived results is that the number of inducing variables required to attain the optimal, minimax contraction rate is sufficiently less in the inverse setting than in the direct problem. Therefore, inverse problems can be scaled up at a higher degree than standard regression models.

\textit{Related literature. } The theory of Bayesian approaches to linear inverse problems is now well established. The study of their asymptotic properties started with the study of conjugate priors \cite{MR3084161, florens2016regularizing,dong:inverse,MR3477780,MR2906881,knapik2013bayesian} before addressing the non-conjugate case \cite{knapik:salomond:2018,kolyan:inverse} and rate-adaptive priors \cite{MR3477780,svv15}. By now we have a good understanding of both the accuracy of the procedure for recovering the true function and the reliability of the corresponding uncertainty statements. The theory of Bayesian non-linear inverse problems is less developed, but recent years have seen an increasing interest in the topic, see the monograph \cite{nickl2022bayesian} and references therein. Some algorithmic developments for variational Gaussian approximations in non-linear inverse problems, and applications to MCMC sampling, can be found in \cite{Pinski2015bis,Pinski2015}. 

The inducing variable approach for GPs proposed by \cite{pmlr-v5-titsias09a,titsias2009variational} has been widely used in practice. Recently, their theoretical behaviour was studied in the direct, nonparametric regression setting. In  \cite{pmlr-v97-burt19a} it was shown that the expected Kullback-Leibler divergence between the variational class and posterior tends to zero when sufficient amount of inducing variables were used. Furthermore, optimal contraction rates \cite{JMLR:v23:21-1128} and  frequentist  coverage guarantees \cite{nieman2022uncertainty,Ray2023,vakili2022improved} were derived for several inducing variable methods. Our paper focuses on extending these results to the linear inverse setting.

\textit{Organization.} The paper is organized as follows. In Section \ref{sec:main} we first introduce the inverse regression model where we carry out our analysis. Then we discuss the Bayesian approach using GPs and its variational approximations in Sections \ref{sec:GP:regression} and \ref{section: VB derivation}, respectively. As our main result we derive contraction rates for general inducing variable methods, both in the mildly and severely ill-posed cases. Then in Section \ref{sec:spectral:feature} we focus on two specific inducing variable methods based on spectral features and provide more explicit results for them. We apply these results for a collection of examples, including the Volterra operator, the heat equation and the Radon transform in Section \ref{sec:examples}. Finally, we demonstrate the applicability of the procedure in the numerical analysis of Section \ref{sec: Numerical analysis} and conclude the paper with discussion in Section \ref{sec:discussion}. The proof of the main theorem together with technical lemmas and additional simulation study are deferred to the supplementary material.

\textit{Notation.} Let  $C,c$ be absolute constants, independent of the parameters of the problem whose values may change from line to line. For two sequences ($a_n$) and ($b_n$) of numbers, $a_n\lesssim b_n$ means that there exists a universal constant $c$ such that $a_n\leq c b_n$ and we write $a_n\asymp b_n$ if both $a_n\lesssim b_n$ and $b_n\lesssim a_n$ hold simultaneously. We denote by $a_n\ll b_n$ if $|a_n/b_n|$ tends to zero. The maximum and minimum of two real numbers $a$ and $b$ are denoted by $a\vee b$ and $a \wedge b$, respectively. We use the standard notation $\delta_{ij}=\mathds{1}_{i=j}$. For $m\geq 1$, we note $\boldsymbol{S}_{++}^m$ the set of positive-definite matrices of size $m\times m$.


\section{Main results}\label{sec:main}

In our analysis we focus on the non-parametric random design regression model where the functional parameter is observed through a linear operator. More formally, we assume to observe i.i.d. pairs of random variables $(x_i,Y_i)_{i=1,...,n}$ satisfying
\begin{equation}\label{eq: model} 
Y_i = \left(\mathcal{A}f_0\right)(x_i) +Z_i,\qquad Z_i\stackrel{iid}{\sim}N(0,\sigma^2),\,  x_i\stackrel{iid}{\sim} G  \qquad i=1,\dots,n,
\end{equation}
where $f_0\in L_2(\mathcal{T};\mu)$, for some domain $\mathcal{T}\subset\mathbb{R}^d$ and measure $\mu$ on $\mathcal{T}$, is the underlying functional parameter of interest and $\mathcal{A}: L_2(\mathcal{T};\mu)\mapsto  L_2(\mathcal{X}; G)$, for the measure  $G=\mathcal{A}\mu$ on $\mathcal{X}$, is a known, injective, continuous linear operator. In the rest of the paper we use the notation $P_{f_0}$ and $E_{f_0}$  for the joint distribution and the corresponding expectation, respectively, of the data $(X,Y)=(x_i,Y_i)_{i=1,...,n}$. Furthermore, we denote by $E_X, P_X, E_{Y| X}, P_{Y| X}$ the expectation/distribution under $G^{\otimes n}$ and the law of $(Y_i)_i$ given the design respectively. Finally, for simplicity we take $\sigma^2=1$ in our computations. 

In the following, denoting $\mathcal{A}^*$ the adjoint of $\mathcal{A}$, we assume that the self-adjoint operator $\mathcal{A}^*\mathcal{A}\colon L_2(\mathcal{T};\mu) \mapsto L_2(\mathcal{T};\mu)$ possesses countably many positive eigenvalues $(\kappa_j^2)_{j}$ with respect to the eigenbasis $(e_j)_j$ (which is verified if $\mathcal{A}$ is a compact operator for instance). We remark that $\left(g_j\right)_{j}$ defined by $\mathcal{A} e_j = \kappa_j g_j$ is an orthonormal basis of $ L_2(\mathcal{X}; G)$. We work on the ill-posed problem where $\kappa_j\to 0$, the rate of decay characterizing the difficulty of the inverse problem. 

\begin{definition}\label{def: illposedness}
We say the problem is mildly ill-posed problem of degree $p>0$ if $\kappa_j\asymp j^{-p}$ has a polynomial decay. In the severely ill-posed problem, the rate we consider is exponential, $\kappa_j\asymp e^{-cj^p}$ for $c>0, p\geq 1$, and $p$ is the degree of ill-posedness once again.
\end{definition}

 In nonparametrics it is typically assumed that $f_0$ belongs to some regularity class. Here we consider the generalized Sobolev space
\begin{equation}\label{eq: sobolev class} \bar{H}^\beta\coloneqq \left\{f\in L_2(\mathcal{T}; \mu):\ \norm{f}_\beta<\infty\right\}, \quad \norm{f}_\beta^2= \sum_j j^{2\beta}\left|\langle f,e_j\rangle\right|^2,\end{equation}
for some $\beta>0$. We note that the difficulty in estimating $f_0$ from the data is twofold: one needs to deal with the observational noise, which is a statistical problem, as well as to invert the operator $\mathcal{A}$, which comes from inverse problem theory. As a result of the ill-posedness of the problem, recovering $f_0$ from the observations may suffer from problems of unidentifiability and instability. The solution to these issues is to incorporate some form of regularization in the statistical procedure. The Bayesian approach provides a natural way to incorporate regularization into the model via the prior distribution on the functional parameter. In fact penalized likelihood estimators can be viewed as the maximum a posteriori estimators with the penalty term induced by a prior. For example Tikhonov type regularizations can be related to the RKHS-norm of a Gaussian Process prior, see \cite{doi:10.1137/0505095,MR2514435} for a more detailed discussion.

\subsection{Gaussian Process priors for linear inverse problems}\label{sec:GP:regression}

We focus on the Bayesian solution of the inverse problem and exploit the Gaussian likelihood structure by considering conjugate Gaussian Process (GP) priors on $f$. A GP $\mathcal{GP}\left(\eta(\cdot),k(\cdot,\cdot)\right)$ is a set of random variables $\left\{f(t)\ |\ t \in \mathcal{T}\right\}$, such that any finite subset follows a Gaussian distribution. The GP is described by the mean function $\eta$ and a covariance kernel $k(t, t')$. We consider centered GPs as priors (i.e. we take $\eta\equiv 0$). Then the bilinear, symmetric nonnegative-definite function $k\colon\ \mathcal{T}\times \mathcal{T}\mapsto \mathbb{R}$ determines the properties of the process (e.g., its regularity). In view of the linearity of the operator $\mathcal{A}$ the corresponding posterior distribution is also a Gaussian process. The mean and covariance function of the posterior is given by
\begin{equation}
\begin{split}\label{eq:posterior}
t&\mapsto K_{t\boldsymbol{\mathcal{A}f}}\left(K_{\boldsymbol{\mathcal{A}f}\boldsymbol{\mathcal{A}f}} + \sigma^2I_n\right)^{-1}\textbf{y},\\
(t,s) &\mapsto k(t,s) - K_{t \boldsymbol{\mathcal{A}f}}\left(K_{\boldsymbol{\mathcal{A}f}\boldsymbol{\mathcal{A}f}} + \sigma^2I_n\right)^{-1}K_{\boldsymbol{\mathcal{A}f}s},
\end{split}
\end{equation}
where $\mathbf{y}=(y_1,\dots,y_n)^T$, $\boldsymbol{\mathcal{A}f}=\big(\mathcal{A}f(x_i)\big)_{i=1,...,n}$, $K_{\boldsymbol{\mathcal{A}f}\boldsymbol{\mathcal{A}f}}=E_{\Pi}\boldsymbol{\mathcal{A}f}\boldsymbol{\mathcal{A}f}^T\in\mathbb{R}^{n\times n}$ with $E_\Pi$ denoting the expectation with respect to the GP prior $\Pi$, $K_{t \boldsymbol{\mathcal{A}f}}^T=\big(E_{\Pi} \mathcal{A}f(x_i) f(t)\big)_{i=1,...,n}\in \mathbb{R}^{n}$, see the supplement for the detailed derivation.

 Due to the closed-form expressions for the posterior and the marginal likelihood, as well as the simplicity with which uncertainty quantification may be produced, GP regression has gained popularity \cite{MR2514435}.  Furthermore, the asymptotic frequentist properties of posteriors corresponding to GP priors in the direct problem, with $\mathcal{A}$ taken to be the identity operator, is well-established by now. Optimal contraction rates and confidence guarantees for Bayesian uncertainty quantification were derived in the regression setting and beyond, see for instance \cite{ic08,patietal15,10.1214/16-AOS1469,10.1214/15-EJS1078,10.1214/19-AOS1881,aadharry11,aadharry08,10.1214/14-AOS1289} and references therein. In the following, we say that $\varepsilon_n$ is an $L_2$--posterior contraction rate for the posterior $\Pi\left[\ \cdot\ | X,Y\right]$ if for any $M_n\to0$
\[ E_{f_0}\Pi\left[f\colon\ \norm{f-f_0}_{L_2(\mathcal{T}; \mu)} \geq M_n\varepsilon_n\ | \ X,Y\right]\to 0.\]

In our analysis we consider covariance kernels with eigenfunctions coinciding with the eigenfunctions of the operator $\mathcal{A}^*\mathcal{A}$, i.e. we take
\begin{equation}\label{eq: kernel decomposition}
k(t,s)=\sum\nolimits_j \lambda_j e_j(t)e_j(s),
\end{equation}
where $(\lambda_j)_{j}$ denote the corresponding eigenvalues. The asymptotic behaviour of the corresponding posterior has been well investigated in the literature both in the mildly and severely ill-posed inverse problems. Rate optimal contraction rates and frequentist coverage guarantees for the resulting credible sets were derived both for known and unknown regularity parameters \cite{florens2016regularizing,MR3477780,MR2906881,knapik2013bayesian,svv15}. These results were further extended for other covariance kernels where the eigenfunctions do not exactly match the eigenfunctions of the operator $\mathcal{A}$, but in principle they have to be closely related, see \cite{MR3084161, dong:inverse,knapik:salomond:2018,kolyan:inverse}.

However, despite the explicit, analytic form of the posterior given in \eqref{eq:posterior} and the theoretical underpinning, the practical applicability of this approach is limited for large sample size $n$. The computation of the posterior involves inverting the $n$-by-$n$ matrix $K_{\boldsymbol{\mathcal{A}f}\boldsymbol{\mathcal{A}f}} + \sigma^2I_n$, which has computational complexity $O(n^3)$. Therefore, in practice often not the true posterior, but a scalable, computationally attractive approximation is applied. Our focus here is on the increasingly popular inducing variable variational Bayes method introduced in \cite{pmlr-v5-titsias09a,titsias2009variational}.

\subsection{Variational GP for linear inverse problems}\label{section: VB derivation}

In variational Bayes the approximation of the posterior is casted as an optimization problem. First a tractable class of distributions $\mathcal{Q}$ is considered, called the variational class. Then the approximation $\Psi^*$ is computed by minimizing the Kullback-Leibler divergence between the variational class and the true posterior, i.e.
\[\Psi^*= \arg\ \inf\nolimits_{Q\in \mathcal{Q}} KL\left(Q|\!| \Pi\left[\ \cdot\ | X,Y\right]\right).\]
 There is a natural trade-off between the computational complexity and the statistical accuracy of the resulting approximation. Smaller variational class results in faster methods and easier interpretation, while more enriched classes preserve more information about the posterior ensuring better approximations.

In context of the Gaussian process regression model (with the operator $\mathcal{A}$ taken to be the identity), \cite{titsias2009variational} proposed a low-rank approximation approach based on inducing variables. The idea is to compress the information encoded in the observations of size $n$ into $m$ so called inducing variables. We extend this idea for linear inverse problems. Let us consider real valued random variables  $\mathbf{u}=(u_1,\dots,u_m)\in L_2\left(\Pi\right)$,  expressed as measurable linear functionals of $f$ and whose prior distribution is $\Pi_u$. In view of the linearity of $\boldsymbol{u}$, the joint distribution of  $(f,\mathbf{u})$ is a Gaussian process, hence the conditional distribution $f|\mathbf{u}$ denoted by $\Pi(\cdot|\mathbf{u})$, is also a Gaussian process with mean function and covariance kernel given by
\begin{align*}
t\mapsto K_{t\boldsymbol{u}} K_{\boldsymbol{u}\boldsymbol{u}}^{-1}\boldsymbol{u}\quad\text{and}\quad (t,s)\mapsto k(t,s)- K_{t\boldsymbol{u}}K_{\boldsymbol{u}\boldsymbol{u}}^{-1}K_{\boldsymbol{u}s},
\end{align*}
respectively, where $K_{t\boldsymbol{u}}=E_\Pi(f(t)\boldsymbol{u})\in \mathbb{R}^m$ and $K_{\boldsymbol{u}\boldsymbol{u}}=E_\Pi(\boldsymbol{u}\boldsymbol{u}^T)\in \mathbb{R}^{m\times m}$. Then the posterior is approximated via a probability measure $\boldsymbol{\Psi}_u$ on $\left(\mathbb{R}^m, \mathcal{B}(\mathbb{R}^m)\right)$ by $\Psi = \int \Pi[\cdot | \mathbf{u}] d\Psi_u( \mathbf{u}),$ which is absolutely continuous against $\Pi$ and satisfies
$ \frac{d\Psi}{d\Pi}(f)=  \frac{d\Psi_u}{d\Pi_u}\left(\mathbf{u}(f)\right).$
We note that the variables $\mathbf{u}$ were first considered to be point evaluations of the GP prior process before these ideas were extended to interdomain inducing variables, e.g. integral forms of the process \cite{NIPS2009_5ea1649a,titsias2009variational}. 

Taking $\Psi_{\boldsymbol{u}}=\mathcal{N}(\boldsymbol{\mu}_u,\Sigma)$ as a multivariate Gaussian, the corresponding $\Psi\propto \Pi(\cdot|\boldsymbol{u})\Psi_{\boldsymbol{u}}$ is a Gaussian process, with mean and covariance functions
\begin{equation}\label{eq:variational:family}
\begin{split}
t\mapsto K_{t\boldsymbol{u}} K_{\boldsymbol{u}\boldsymbol{u}}^{-1}\boldsymbol{\mu}_u,\quad \text{and}\quad
(t,s)\mapsto k(t,s)- K_{t\boldsymbol{u}}K_{\boldsymbol{u}\boldsymbol{u}}^{-1}(K_{\boldsymbol{u}\boldsymbol{u}}-\Sigma)K_{\boldsymbol{u}\boldsymbol{u}}^{-1} K_{s\boldsymbol{u}}^T.
\end{split}
\end{equation}
Letting $\boldsymbol{\mu}_u$ and $\Sigma$ be the free variational parameters, the variational family is taken as 
\[\mathcal{Q}\coloneqq\left\{\Psi\ |\ \Psi_u=\mathcal{N}(\boldsymbol{\mu}_u,\Sigma_u),\ \boldsymbol{\mu}_u\in\R^m,\ \Sigma_u\in \boldsymbol{S}_{++}^m\right\},\] 
consisting of ``$m$--sparse'' Gaussian processes. 

By similar computations as those from \cite{JMLR:v23:21-1128}, it can be shown that $\Pi\left[\ \cdot\ | X,Y\right]$ is equivalent to any element of $\mathcal{Q}$ (they are mutually dominated) so that the $KL$ divergence is always finite and there exists a $\Psi_u^*$, corresponding to the minimizer $\Psi^*$ of $\text{KL}\left(\Psi|\! |\Pi[\cdot | X,Y]\right)$. Furthermore, we have
\begin{align*}
 \frac{d\Psi^*}{d\Pi}(f)= \frac{d\Psi_{\boldsymbol{u}}^*}{d\Pi_{\boldsymbol{u}}}(\mathbf{u}) &\propto exp\Big(-\frac{1}{2\sigma^2} \int \sum\nolimits_{i=1}^n (Y_i-\mathcal{A}f(x_i))^2 d\Pi(f|\mathbf{u}) \Big)\\
 &\propto exp\Big(-\frac{1}{2\sigma^2} \sum\nolimits_{i=1}^n (Y_i-K_{\mathcal{A}f(x_i)\boldsymbol{u}}K_{\boldsymbol{u}\boldsymbol{u}}^{-1}\mathbf{u})^2 \Big)
\end{align*}
where $K_{\mathcal{A}f(x_i)\boldsymbol{u}}=E_{\Pi}\mathcal{A}f(x_i)\boldsymbol{u}^T$. One can observe that the parameters of the variational approximations are 
\begin{equation}\label{eq:variational:approx}
\begin{split}
\boldsymbol{\mu}_u^*&=\sigma^{-2}K_{\boldsymbol{u}\boldsymbol{u}}\left(K_{\boldsymbol{u}\boldsymbol{u}}+\sigma^{-2}K_{\boldsymbol{u}\boldsymbol{\mathcal{A}f}}K_{\boldsymbol{\mathcal{A}f}\boldsymbol{u}}\right)^{-1}K_{\boldsymbol{u}\boldsymbol{\mathcal{A}f}}\mathbf{y},\\
\Sigma^*_u&=K_{\boldsymbol{u}\boldsymbol{u}}\left(K_{\boldsymbol{u}\boldsymbol{u}}+\sigma^{-2}K_{\boldsymbol{u}\boldsymbol{\mathcal{A}f}}K_{\boldsymbol{u}\boldsymbol{\mathcal{A}f}}\right)^{-1}K_{\boldsymbol{u}\boldsymbol{u}},
\end{split}
\end{equation}
for $K_{\boldsymbol{u}\boldsymbol{\mathcal{A}f}}=K_{\boldsymbol{\mathcal{A}f}\boldsymbol{u}}^T=E_\Pi\mathbf{u}(\mathbf{\mathcal{A}f})^T$ the $m\times n$ matrix whose $j$th column is $K_{\mathcal{A}f(x_i)\boldsymbol{u}}$ and $(\mathbf{\mathcal{A}f})^T=\left(\mathcal{A}f(x_1),\dots,\mathcal{A}f(x_n)\right)$. Then the explicit form for the variational posterior $\Psi^{*}$ can be attained by plugging in the parameters \eqref{eq:variational:approx} into the variational mean and covariance function \eqref{eq:variational:family}. We also define $Q_{\boldsymbol{\mathcal{A}f}\boldsymbol{\mathcal{A}f}}=K_{\boldsymbol{u}\boldsymbol{\mathcal{A}f}}^TK_{\boldsymbol{u}\boldsymbol{u}}K_{\boldsymbol{u}\boldsymbol{\mathcal{A}f}}$. 

We investigate the statistical inference properties of the above variational posterior distribution $\Psi^*$. More concretely we focus on how well the variational approximation can recover the underlying true functional parameter $f_0$ of interest in the indirect, linear inverse problem \eqref{eq: model}. We derive contraction rate for $\Psi^*$ both in the mildly and severely ill-posed inverse problem case. Furthermore, we consider both  the standard exponential and polynomial spectral structures for the prior, i.e. we assume that the eigenvalues of the prior covariance kernel satisfies either  $\lambda_j\asymp j^{-\alpha}e^{-\xi j^p}$ or $\lambda_j\asymp j^{-1-2\alpha}$ for some $\alpha\geq0,\ \xi>0$.  Finally, in view of \cite{JMLR:v23:21-1128}, we introduce additional assumptions on the covariance kernel of the conditional distribution of $f|\boldsymbol{u}$  ensuring that the variational posterior is not too far from the true posterior in Kullback-Leibler divergence. 

\begin{theorem}\label{th: post rates}
Let's assume that $f_0\in \bar{H}^\beta$ and $\norm{f_j}_{\infty} \lesssim j^\gamma$ for $\beta>0,\gamma\geq0$.
\begin{enumerate}
\item\label{pol decay eig values} In the mildly-ill posed problem where $\kappa_j\asymp j^{-p}$, $p>0$, if $\lambda_j\asymp j^{-1-2\alpha}$ for $\alpha>0$ and $(\alpha \wedge \beta)+p > 3/2 + 2\gamma$, the posterior contracts at the rate $\varepsilon_n^{\text{inv}}=n^{-\frac{\alpha\wedge\beta}{1+2\alpha+2p}}$.
\item\label{exp decay eig values} In the severely ill-posed problem where $\kappa_j\asymp e^{-cj^p}$, $c>0,p\geq1$, if $\lambda_j\asymp j^{-\alpha}e^{-\xi j^p}$ for $\alpha\geq 0$, $\xi>0$, the posterior contracts at the rate $\varepsilon_n^{\text{inv}}=\log^{-\beta/p} n$.
\end{enumerate}
Furthermore, if there exists a constant $C$ independent of $n$ such that
\begin{equation}\label{eq:cond:K-Q}
E_X \norm{K_{\boldsymbol{\mathcal{A}f}\boldsymbol{\mathcal{A}f}}-Q_{\boldsymbol{\mathcal{A}f}\boldsymbol{\mathcal{A}f}}}\leq C,\quad\text{and}\quad E_X Tr\left(K_{\boldsymbol{\mathcal{A}f}\boldsymbol{\mathcal{A}f}}-Q_{\boldsymbol{\mathcal{A}f}\boldsymbol{\mathcal{A}f}}\right) \leq Cn\varepsilon_n^2,
\end{equation}
where $\varepsilon_n = n^{-\frac{\alpha\wedge\beta+p}{1+2\alpha+2p}}$ in \ref{pol decay eig values}., and $\varepsilon_n = n^{-c/(\xi+2c)} \log^{-\beta/p+c\alpha/(\xi+2c)} (n)$ in \ref{exp decay eig values}., $\Psi^*$ contracts around $f_0$ at the rate $\varepsilon_n^{\text{inv}}$ for the mildly and severely ill-posed problems i.e.
\[ E_{f_0}\Psi^*\left[f\colon\ \norm{f-f_0}_{L_2(\mathcal{T}; \mu)} \geq M_n\varepsilon_n^{\text{inv}}\right]\to 0,\quad M_n\to \infty.\]
\end{theorem}

\begin{proof}
We provide the sketch of the proof here, the detailed derivation of the theorem is deferred to the supplementary material. In a first step, we derive posterior contraction rates around $\mathcal{A}f_0$ in empirical $L_2$-norm under fixed design. In particular, we obtain an exponential decay of the probability expectation in the form
\begin{equation}
E_{Y\given X}\Pi\Big[f\colon\ n^{-1}\sum\nolimits_{i=1}^n \big(\mathcal{A}f-\mathcal{A}f_0\big)^2(x_i) \geq M_n\varepsilon_n^2\given\ X, Y\Big]\mathds{1}_{A_n}\leq Ce^{-cM_n^2n\varepsilon_n^2},\label{eq:contraction:direct}
\end{equation}
for arbitrary $M_n\to \infty$, where $\varepsilon_n=n^{-\frac{\alpha\wedge\beta+p}{1+2\alpha+2p}}$ in the mildly and $\varepsilon_n=n^{-\frac{c}{\xi+2c}}\log^{-\frac{\beta}{p}+\frac{c\alpha}{\xi+2c}} n$ in the severely ill-posed problems and $A_n$ is an event on the sample space $\mathbb{R}^n$ with probability tending to one asymptotically. In the mildly ill-posed case, this follows from results in \cite{MR2332274, aadharry08}, while additional care is needed in the severely ill-posed case. As a second step, we go back to the random design setting. We show, using concentration inequalities and controlling the tail probability of GPs in the spectral decomposition, that the empirical and population $L_2$-norms are equivalent on a large enough event. This implies 
 contraction rate with respect to the $\|\cdot\|_{L^{2}(\mathcal{X},G)}$-norm around $\mathcal{A}f_0$, similarly to \eqref{eq:contraction:direct}.  In the third step, using the previous result on the forward map, we derive contraction rates around $f_0$. To achieve this we apply the modulus of continuity techniques introduced in \cite{knapik:salomond:2018}. Notably, we extend their ideas to infinite Gaussian series priors in the severely ill-posed case as well. Since in all these steps we can preserve the exponential upper bound for the posterior contraction (on a large enough event), we can apply Theorem 5 of \cite{ray2022variational}, resulting in contraction rates for the VB procedure. It requires a control of the expected KL divergence between these two distributions, which follows from our assumptions on the expected trace and spectral norm of the covariance matrix of $\boldsymbol{\mathcal{A}f}|\boldsymbol{u}$, see Lemma 3 in \cite{JMLR:v23:21-1128} for the identity operator $\mathcal{A}$.

\end{proof}

We briefly discuss the above results. First of all, the $L_2(\mathcal{T}; \mu)$-contraction rate of the true posterior, to the best of our knowledge, wasn't derived explicitly in the literature before, hence it is of interest in its own right. Nevertheless, the main message is that the variational posterior achieves the same contraction rate as the true posterior under the assumption \eqref{eq:cond:K-Q}. Note that in the mildly ill-posed inverse problem case for eigenvalues $\lambda_j\asymp j^{-1-2\beta}$ (i.e. taking $\alpha=\beta$), the posterior contracts with the minimax rate $n^{-\beta/(1+2\beta+2p)}$. Note that the $d$--dimensional case directly follows from this result when one defines the regularity class \eqref{eq: sobolev class} and ill-posedness (Definition \ref{def: illposedness}) with $\beta/d$ and $p/d$ which would imply the rate $n^{-\beta/(d+2\beta+2p)}$. Similarly in the severely-ill posed case one can achieve the minimax logarithmic contraction rate. Furthermore, the choice of the eigenvalue structure in the theorem was done for computational convenience, the results can be generalised for other choices of $\lambda_j$ as well. Though we considered the random variables $u$ fixed as we do not optimize them above, they could conceivably be considered as free variational parameters and selected at the same time as $\boldsymbol{\mu}_u$ and $\Sigma_u$. 

In the next subsection we consider two specific choices of the inducing variables, i.e. the population spectral feature method and its empirical counter part. We show that under sufficient condition on the number of inducing variables condition \eqref{eq:cond:K-Q} is satisfied implying the contraction rate results derived in the preceding theorem.

\subsection{Population and empirical spectral features methods}\label{sec:spectral:feature}
We focus here on two inducing variables methods, based on the spectral features (i.e. eigenspectrum) of the empirical covariance matrix $K_{\boldsymbol{\mathcal{A}f}\boldsymbol{\mathcal{A}f}}=E_{\Pi}\mathbf{\mathcal{A}f}\mathbf{\mathcal{A}f}^T$ and the corresponding population level covariance operator $(x,y)\mapsto E_{\Pi}\mathcal{A}f(x)\mathcal{A}f(y)$. 

We start with the former method and consider inducing variables of the form
\begin{equation}\label{eq: inducing var matrix}
 u_j = \sum\nolimits_{i=1}^n v_j^i \mathcal{A}f(x_i),\quad j=1,\dots,m,
 \end{equation}
 where $\mathbf{v}_j=(v_j^1,\dots,v_j^n)$ is the eigenvector of $K_{\boldsymbol{\mathcal{A}f}\boldsymbol{\mathcal{A}f}}$ corresponding to the $j$th largest eigenvalue $\rho_j$ of this matrix. Similarly to the direct problem studied in \cite{pmlr-v97-burt19a, JMLR:v23:21-1128}, this results in $\left(K_{\boldsymbol{\mathcal{A}f}\boldsymbol{\mathcal{A}f}}\right)_{ij}=\rho_j \delta_{ij}$, $\left(K_{\boldsymbol{\mathcal{A}f}\boldsymbol{u}}\right)_{ij}=\rho_j v_j^i$, $Q_{\boldsymbol{\mathcal{A}f}\boldsymbol{\mathcal{A}f}}=\sum_{j=1}^m\rho_j\mathbf{v}_j\mathbf{v}_j^T$, and $K_{\boldsymbol{\mathcal{A}f}\boldsymbol{\mathcal{A}f}}-Q_{\boldsymbol{\mathcal{A}f}\boldsymbol{\mathcal{A}f}}=\sum_{j=m+1}^n\rho_j\mathbf{v}_j\mathbf{v}_j^T$. The computational complexity of deriving the first $m$ eigenvectors of  $K_{\boldsymbol{\mathcal{A}f}\boldsymbol{\mathcal{A}f}}$ is $\mathcal{O}(n^2m)$. This is still quadratic in $n$, which sets limitations to its practical applicability, but it can be computed for arbitrary choices of the prior covariance operator and map $\mathcal{A}$. We also note that this choice gives the optimal rank--$m$ approximation $Q_{\boldsymbol{\mathcal{A}f}\boldsymbol{\mathcal{A}f}}$ of $K_{\boldsymbol{\mathcal{A}f}\boldsymbol{\mathcal{A}f}}$ and it was noted in \cite{pmlr-v97-burt19a} that it gives the minimiser of the trace and norm terms in \eqref{eq:cond:K-Q}.

The second inducing variables method is based on the eigendecomposition of covariance kernel $(x,y)\mapsto E_{\Pi}\mathcal{A}f(x)\mathcal{A}f(y)$. Let us consider the variables
\begin{equation}\label{eq: inducing var op}
 u_j = \int_\mathcal{X} \mathcal{A}f(x)e_j(x)dG(x),\quad j=1,\dots,m.
 \end{equation}
Again, by extending the results derived in the direct problem \cite{pmlr-v97-burt19a} to the inverse setting, this results in $\left(K_{\boldsymbol{\mathcal{A}f}\boldsymbol{\mathcal{A}f}}\right)_{ij}=\lambda_j\kappa_j \delta_{ij}$, $\left(K_{\boldsymbol{\mathcal{A}f}\boldsymbol{u}}\right)_{ij}=\lambda_j\kappa_j \boldsymbol{\phi}_j^i$, $Q_{\boldsymbol{\mathcal{A}f}\boldsymbol{\mathcal{A}f}}=\sum_{j=1}^m\lambda_j\kappa_j \boldsymbol{\phi}_j\boldsymbol{\phi}_j^T$, and $K_{\boldsymbol{\mathcal{A}f}\boldsymbol{\mathcal{A}f}}-Q_{\boldsymbol{\mathcal{A}f}\boldsymbol{\mathcal{A}f}}=\sum_{j=m+1}^n\lambda_j\kappa_j\boldsymbol{\phi}_j\boldsymbol{\phi}_j^T$, where $\boldsymbol{\phi}_j=\left(\phi_j(x_1),\dots,\phi_j(x_n)\right)^T$. The computational complexity of this method is $O(nm^2)$, which is substantially faster than its empirical counter part. However, it requires the exact knowledge of the eigenfunctions of the prior covariance kernel, and therefore in general has limited practical applicability.

\begin{corollary}\label{lemma: number of inducing variables}
Let's assume that $f_0\in \bar{H}^\beta$, $\norm{g_j}_{\infty} \lesssim j^\gamma$ for $\beta>0,\gamma\geq0$ and in the
\begin{enumerate}
\item\label{pol decay eig values 2}\label{mild case} mildly-ill posed case $\kappa_j\asymp j^{-p}$: take prior eigenvalues $\lambda_j\asymp j^{-1-2\alpha}$ for some $\alpha>0$, $(\alpha \wedge \beta)+p > 3/2 + 2\gamma$, number of inducing variables $m_n\geq  n^{\frac{1}{(1+2p+2\alpha)}}$ and denote by $\varepsilon_n^{\text{inv}}=n^{-\frac{\alpha\wedge \beta}{1+2\alpha+2p}}$.
\item\label{exp decay eig values 2} severely ill-posed case $\kappa_j\asymp e^{-cj^p}$: take prior eigenvalues $\lambda_j\asymp j^{-\alpha}e^{-\xi j^p}$ for $\alpha\geq 0$, $\xi>0$,  number of inducing variables  $m_n^p\geq \big(\xi+2c\big)^{-1}\log n$, and introduce the notation $\varepsilon_n^{\text{inv}}=\log^{-\beta/p} n$.
\end{enumerate}
Then both for the population (if $\gamma=0$ in \ref{mild case}.) and empirical spectral features variational methods the corresponding variational posterior distribution contracts around the truth with the rate $\varepsilon_n^{\text{inv}}$, i.e.
\[ E_{f_0}\Psi^*\left[f\colon\ \norm{f-f_0}_{L_2(T; \mu)} \geq M_n\varepsilon_n^{\text{inv}}\right]\to 0,\quad M_n\to \infty.\]
\end{corollary}

\begin{remark}
In the mildly ill-posed inverse problem taking $\alpha=\beta$ results in the minimax contraction rate for $m_n\geq  n^{\frac{1}{1+2p+2\alpha}}$. Note that it is substantially less compared to the direct problem with $p=0$, hence the computation is even faster in the inverse problem case.
\end{remark}


\section{Examples}\label{sec:examples}
In this section we provide three specific linear inverse problems as examples. The Volterra (integral) operator and the Radon transformations are mildly ill-posed, while the heat-equation is a severely ill-posed inverse problem. We show that in all cases by optimally tunning the GP prior and including enough inducing variables, the variational approximation of the posterior provides (from a minimax perspective) optimal recovery of the underlying signal $f_0$.
\subsection{Volterra operator}\label{subsec: volterra}

First, let us consider the Volterra operator, $\mathcal{A}:  L_2[0,1]\longrightarrow L_2[0,1]$ satisfying that
\begin{equation}\label{eq: volterra operator} \mathcal{A}f(x) = \int_0^x f(s)ds,\quad \mathcal{A}^*f(x) = \int_x^1 f(s)ds.\end{equation}
The eigenvalues of $\mathcal{A}^*\mathcal{A}$ and the corresponding eigenbases are given by $ \kappa^2_j = (j-1/2)^{-2}\pi^{-2},\ e_j(x)=\sqrt{2} \cos\left((j-1/2)\pi x\right),\ g_j(x)=\sqrt{2} \sin\left((j-1/2)\pi x\right)$ respectively,
see \cite{halmos2012hilbert}. Therefore the problem is mildly ill-posed with degree $p=1$ and these bases are uniformly bounded, i.e. $\sup_j \|e_j\|_{\infty}\vee \|g_j\|_{\infty}<\infty$. The following lemma is then a direct application of Corollary \ref{lemma: number of inducing variables}.

\begin{corollary}\label{cor: volterra}
Consider the Volterra operator in \eqref{eq: model} and assume that $f_0\in \bar{H}^\beta$, for some $\beta>1/2$. Set the eigenvalues in \eqref{eq: kernel decomposition} as $\lambda_j = j^{-1-2\beta}$. Then the variational posterior $\Psi^*$ resulting from either the empirical or population spectral features inducing variable methods achieves the minimax contraction rate if the number of inducing variables exceeds $m_n\gtrsim n^{\frac{1}{3+2\beta}}$, i.e. for arbitrary $ M_n\to\infty$
\[ E_{f_0} \Psi^*\left[\norm{f-f_0}_{L_2[0,1]} \geq M_n n^{-\beta/(3+2\beta)}\right] \to 0.\]
\end{corollary}

\subsection{Heat equation}\label{subsec: heat}
Next let us consider the problem of recovering the initial condition for the heat equation. The heat equation is often considered as the starting example in the PDE literature and, for instance, the Black-Scholes PDE can be converted to the heat equation as well. We consider the Dirichlet boundary condition
\begin{equation}\label{eq: heat equation} \frac{\partial}{\partial t}u(x,t)= \frac{\partial^2}{\partial x^2}u(x,t),\quad u(x,0)=\mu(x),\quad u(0,t)=u(1,t)=0,\end{equation}
for $u$ defined on $\left[0,1\right]\times\left[0,T\right]$, $T>0$. For $\mu\in L_2[0,1]$, 
$ u(x,t)=\sqrt{2}\sum_{j=1}^\infty \mu_j e^{-j^2\pi^2 t}\sin(j\pi x),$
with $\mu_j=\sqrt{2}\int_0^1 \mu(s)\sin(j\pi s)ds$. Therefore, if $\mathcal{A}: \mathcal{D}\mapsto \mathcal{D}$, with $\mathcal{D}\coloneqq \left\{f\in L_2[0,1],\ f(0)=f(1)=0\right\}$, is such that, for $\mu=f$, $\mathcal{A}f(x) = u(x,T)$, then the corresponding singular-values and singular-functions of the operator $\mathcal{A}$ are $\kappa_j=e^{-j^2\pi^2 T}$ and $e_j(x)=g_j(x)=\sqrt{2}\sin(j\pi x)$. Therefore it is a severely ill-posed problem with $p=2$ and $c=\pi^2 T$. We also note that $\sup_{j}\|e_j\|_{\infty}\vee \|g_j\|_{\infty}<\infty$. This problem has been well studied both in the frequentist \cite{bissantz2008statistical,golubev1999statistical,mair1996statistical} and Bayesian setting \cite{knapik2013bayesian,stuart2010inverse}. Then, by direct application of Corollary \ref{lemma: number of inducing variables} we can provide optimality guarantees for the variational Bayes procedure in this model as well.

\begin{corollary}\label{cor:heat}
Consider the heat equation operator $\mathcal{A}$ as above in the linear inverse regression model \eqref{eq: model} and assume that $f_0\in \bar{H}^\beta$ for some $\beta>0$. Furthermore, we set the eigenvalues  $\lambda_j =  j^{-\alpha}e^{-\xi j^2}$, $\alpha\geq0$, $\xi>0$ in \eqref{eq: kernel decomposition}. Then the variational approximation $\Psi^*$ resulting from either of the spectral features inducing variables method with $m_n\geq \left(\xi + \pi^2 T \right)^{-1/2}\log^{1/2} n$ achieves the minimax contraction rate, i.e. for arbitrary $M_n\to\infty$
\[ E_{f_0} \Psi^*\left[\norm{f-f_0}_{L_2([0,1])} \geq M_n \log^{-\beta/2} n\right] \to 0.\]
\end{corollary}

\subsection{Radon transform}\label{subsec: radon}

Finally, we consider the Radon transform \cite{MR1041393}, where for some (Lebesgue)--square-integrable function $f\ \colon\ D \to \mathbb{R}$ defined on the unit disc $D=\left\{x\in \mathbb{R}^2:\ \norm{x}_2\leq 1\right\}$, we observe its integrals along any line intersecting $D$. If we parameterized the lines by the length $s\in[0,1]$ of their perpendicular from the origin and the angle $\phi\in[0.2\pi)$ of the perpendicular to the x-axis, we observe
\begin{equation}\label{eq: Radon transform}
 \mathcal{A}f(s,\phi)=\frac{\pi}{2\sqrt{1-s^2}}\int_{-\sqrt{1-s^2}}^{\sqrt{1-s^2}} f(s\cos\phi -t\sin\phi, s\sin\phi+t\cos\phi)dt,
 \end{equation}
where $(s,\phi)\in S=[0,1]\times [0,2\pi)$. The Radon transform is then a map from $\mathcal{A}\colon L_2(D;\mu)\to L_2(S;G)$, where $\mu$ is $\pi^{-1}$ times the Lebesgue measure and $dG(s,\phi)=2\pi^{-1}\sqrt{1-s^2}dsd\phi$. Then $\mathcal{A}$ is a bijective, mildly ill-posed  linear operator of order $p=1/4$. Furthermore, the operator's singular value decomposition can be computed via Zernike polynomials $Z_m^k$ (degree $m$, order $k$) and Chebyshev polynomials of the second kind $U_m(\cos \theta)= \sin\left((m+1)\theta\right)/\sin \theta\leq m+1$, see \cite{MR1041393}. Translating it to the single index setting, we get  for some functions $l,m: \mathbb{N}\mapsto \mathbb{N}$ satisfying $m(j)\asymp \sqrt{j}$ and $|l(j)|\leq m(j)$, that
\[ e_j(r,\theta)=\sqrt{m(j)+1}Z_{m(j)}^{|l(j)|} e^{j l(j)\theta}, \quad g_j(s,\phi)=U_{m(j)}(s) e^{j l(j)\phi},\]
if polar coordinates are used on $D$. Therefore, we have $\sup_{j} (\|e_j\|_{\infty}\vee \|g_j\|_{\infty})/\sqrt{j}<\infty$. Then, by directly applying Corollary \ref{lemma: number of inducing variables} to this setting we can show that the variational Bayes method achieves the minimax contraction rate.
\begin{corollary}
Consider the Radon transform operator \eqref{eq: Radon transform} in the inverse regression model \eqref{eq: model} and let us take $f_0\in \bar{H}^\beta$, $\beta>9/4$. Taking polynomially decaying eigenvalues $\lambda_j\asymp j^{-1-2\beta}$, the empirical spectral features variational Bayes method achieves the optimal minimax contraction rate if $m_n\gtrsim n^{1/(3/2+2\beta)}$, i.e. for any $M_n\to\infty$
\[ E_{f_0} \Psi^*\left[\norm{f-f_0}_{L_2(D;\mu)} \geq M_n n^{-\beta/(3/2+2\beta)}\right] \to 0.\]
\end{corollary}


\section{Numerical analysis}\label{sec: Numerical analysis}

We demonstrate the approximation accuracy of the variational Bayes method on synthetic data. We consider here the recovery of the initial condition of the heat condition \ref{subsec: heat}, which is a severely ill-posed. In the supplement we provide additional simulation study for mildly ill-posed inverse problems as well.
We set the sample size $n=8000$, take uniformly distributed covariates on $[0,1)$, and let 
\[ f_0(t) =\sqrt{2}  \sum\nolimits_j c_j j^{-(1+\beta)}\sin(j \pi t),\quad c_j = \begin{cases} 1+0.4 \sin(\sqrt{5}\pi j), & j \text{ odd,}  \\ 
2.5+2 \sin(\sqrt{2}\pi j), & j \text{ even,} \end{cases}\]
for $\beta=1$. The independent- observations are generated as $Y_i\sim \mathcal{N}(\mathcal{A}f_0(x_i), 1)$, depending on the solution of the forward map $\mathcal{A}f_0$ after time $T=10^{-2}$.

We consider the prior with $\lambda_j = e^{-\xi j^2}$ for $\xi=10^{-1}$. In view of Corollary \ref{cor:heat} the optimal number of inducing variables is $m=\big(\xi + 2\pi^2 T\big)^{-1/2}\log^{1/2} n \approx 6$. We consider the population spectral feature method described in \eqref{eq: inducing var op} and plot the variational approximation of the posterior for $m=6$ and $m=3$ inducing variables in Figure \ref{fig: signal}. We represent the true posterior mean by solid red and the upper and lower pointwise $2.5\%$ quantiles by dashed red curves. The true function is given by blue and the mean and quantiles of the variational approximation by solid and dotted purple curves, respectively. 

Observe that with $m=6$, see left part of Figure \ref{fig: signal}, the variational approximation results in similar 95\% pointwise credible bands and posterior mean as the true posterior, providing an accurate approximation. Also note that both the true and the variational posterior contain $f_0$ at most of the points, indicating frequentist confidence validity of the set. At the same time, by taking a factor of two less inducing points, i.e. $m=3$, the credible sets will be overly large, resulting in good frequentist coverage, but suboptimally large posterior spread, see the second plot in Figure \ref{fig: signal}. 

The computations were carried out with a 2,6 GHz Quad-Core Intel Core i7 processor. The computation of the exact posterior mean and covariance kernel on a grid of $100$ points took over half an hour (in CPU time), while the variational approximation was substantially faster, taking only $50.5$ ms, resulting in a $3.68*10^{4}$ times speed.

\begin{figure}[h!]
\center
  \includegraphics[width=\textwidth]{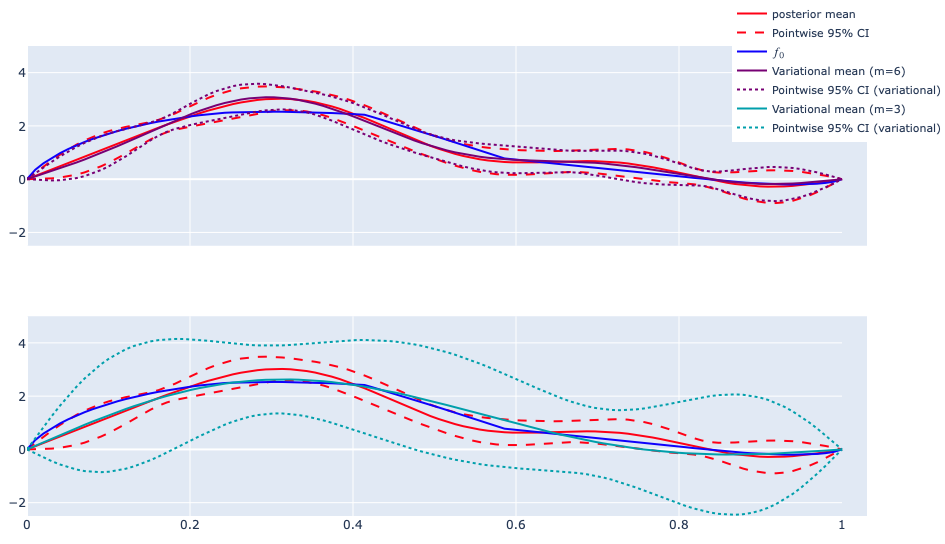}
  \caption{True and variational posterior means and credible regions for Gaussian series prior (sine basis) on the initial condition $\mu=f_0$ of the heat equation \eqref{eq: heat equation}, for $m = 6$ (left) or $m=3$ (right) inducing variables from method \eqref{eq: inducing var op}.}
  \label{fig: signal}
\end{figure}

A more extensive numerical analysis is available in the appendix, considering the application of our method to the settings of Sections \ref{subsec: volterra} and \ref{subsec: radon} as well. We conduct these experiments several times and compare the average Mean Integrated Squared Error (MISE), see appendix A, and compute time for different choices of $m$. We observe that in all our examples, while increasing $m$ results in longer computation, the MISE does not improve after a threshold close to the one presented in our results. Therefore, it is sufficient to include as many inducing variables as we considered in Corollary \ref{lemma: number of inducing variables} to obtain better performance. More than that would would only increase the computation complexity. In the Appendix, we also provide a literature review and some justifications of how relevant these problems are in practice

\section{Discussion}\label{sec:discussion}

We have extended the inducing variables variational Bayes method for linear inverse problems and derived asymptotic contraction rate guarantees for the corresponding variational posterior. Our theoretical results provide a guide for practitioners on how to tune the prior distribution and how many inducing variables to apply (in the spectral feature variational Bayes method) to obtain minimax rate optimal recovery of the true functional parameter of interest. We have demonstrated the practical relevance of this guideline numerically on synthetic data and have shown that using less variables results in highly suboptimal recovery. 

In our analysis we have considered priors built on the singular basis of the operator $\mathcal{A}$. In principle our results can be extended to other priors as well, until the eigenbasis of the covariance operator is not too different from the basis of the operator $\mathcal{A}$. This, however, would complicate the computation of the Kullback -Leibler divergence between the variational family and the posterior, resulting in extra technical challenges. In this setting the empirical spectral features method seems practically more feasible, especially, if the eigenbasis of the covariance kernel is not known explicitly. In the literature several different types of inducing variable methods were proposed, considering other, practically more relevant approaches of interest. Furthermore, extension to other type of inverse problems is also feasible. For instance in the deconvolution problem, when $f_0$ is convoluted with a rectangular kernel, the eigenvalues given by the SVD are the product of a polynomially decaying and oscillating part and the ``average degree'' of ill-posedness does not match the lower and upper bounds \cite{MR2102493}. Extension to non-linear inverse problem is highly relevant, as these problems tend to be computationally even more complex, but very challenging. One possible approach is to linearize the problem and take its variational approximation. Finally, it is of importance to derive frequentist coverage guarantees for VB credible sets. Our approach cannot directly be extended for this task. However, in the direct case, for some special choices of inducing variables, frequentist coverage guarantees were derived using kernel ridge regression techniques \cite{nieman2022uncertainty,vakili2022improved}. This result, although computationally somewhat cumbersome, in principle can be extended to the inverse setting as well. One last drawback of our results is that the priors we consider are non-adaptative in the mildly ill-posed case. Minimax contraction rates are attainable only if the covariance eigenvalues are properly tuned, given the smoothness $\beta$. While this is not an issue the severely ill-posed case in our results, we keep the study of adaptation for future works as it is a much more involved question.

\textit{Funding.} Co-funded by the European Union (ERC, BigBayesUQ, project number: 101041064). Views and opinions expressed are however those of the author(s) only and do not necessarily reflect those of the European Union or the European Research Council. Neither the European Union nor the granting authority can be held responsible for them.

\newpage

\appendix

\setcounter{page}{1}


\section{Additional experiments}

\subsection{Heat equation}

Pursuing the study of the recovery of the initial condition of the heat equation of Section \ref{sec: Numerical analysis}, we repeat the experience $50$ times with $n=4000$ observations, considering all other parameters identical to those used before. According to our theory, the number of inducing variables we should use is still equal to $m=\big(\xi + 2\pi^2 T\big)^{-1/2}\log^{1/2} n \approx 6$.  As before, we consider the population spectral feature method described in \eqref{eq: inducing var op}.

The results from one experiment are presented in Figure \ref{fig: Heat 4000}. We plot the resulting variational approximation of the posterior for $m=6$ and $m=3$ inducing variables and represent the true posterior mean by solid red and the upper and lower pointwise $2.5\%$ quantiles by dashed red curves. The true function is given by blue and the mean and quantiles of the variational approximation by solid and dotted purple/cyan curves, respectively. 

The conclusions we draw from this experiment are the same as those in Section \ref{sec: Numerical analysis}. With the optimal choice $m=6$ following from our theoretical results(on the left of Figure \ref{fig: Heat 4000}), the posterior and variational means are almost indistinguishable and the 95\% pointwise credible bands are identical and contains $f_0$ almost everywhere. However, reducing the number of inducing points to $m=3$, the variational credible sets become much larger, providing less information about $f_0$, and the variational posterior mean is smoother, providing a worse fit to $f_0$.

\begin{figure}[h!]
\center
  \includegraphics[width=\textwidth]{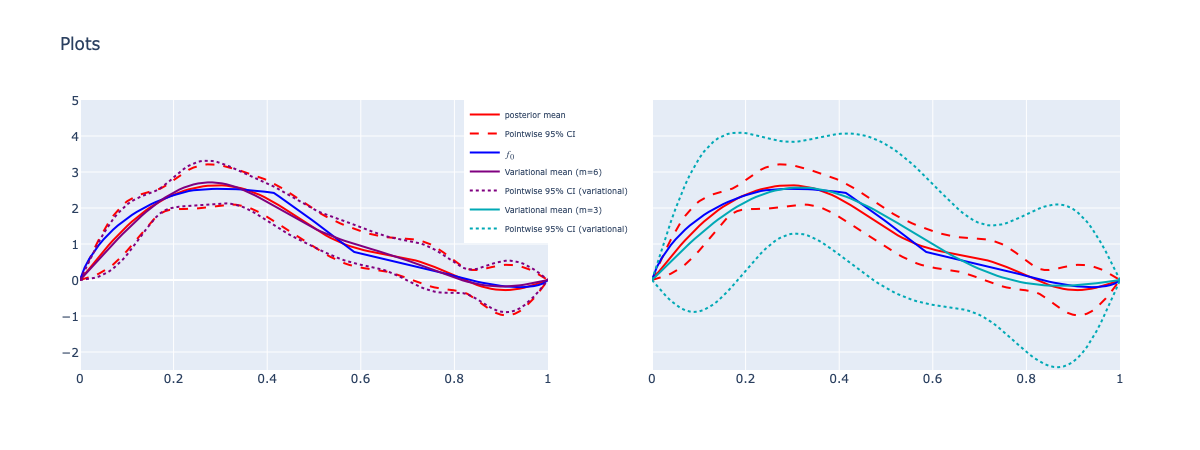}
  \caption{True and variational posterior means and credible regions for Gaussian series prior (sine basis) on the initial condition $\mu=f_0$ of the heat equation \eqref{eq: heat equation}, with $m = 6$ (left) or $m=3$ (right) inducing variables from method \eqref{eq: inducing var op}, computed from $n=4000$ observations.}
  \label{fig: Heat 4000}
\end{figure}

In Figure \ref{fig: box heat}, we summarize the results from the $50$ experiments we ran, assessing the quality of the different posterior distributions we consider via the mean integrated squared error (MISE)
\begin{equation}\label{eq: MISE} \int \norm{f-f_0}^2_{L_2(\mathcal{T};\mu)} d\Pi[f| X,Y],\end{equation}
which can be computed explicitly as the posterior is Gaussian. We compare the true posterior and the variational posteriors obtained with the optimal choice of $m=6$ inducing variables and twice more/less variables with $m=12$ and $m=3$, respectively. On the right-hand side of Figure \ref{fig: box heat}, we see that $m=3$ is a suboptimal choice as it results in a much higher MISE than the other approaches. On the left-hand side of Figure \ref{fig: box heat}, we also report the computation times of the methods, and we highlight that the true posterior takes much longer than any of the variational approximations. On Figure \ref{fig: box heat zoom}, we further see that increasing the number of inducing variables results in more computation time, as expected. At the same time, increasing the number of inducing variables beyond the optimal threshold ($m=6$) does not increase the accuracy considerably.

\begin{figure}[h!]
\center
  \includegraphics[width=\textwidth]{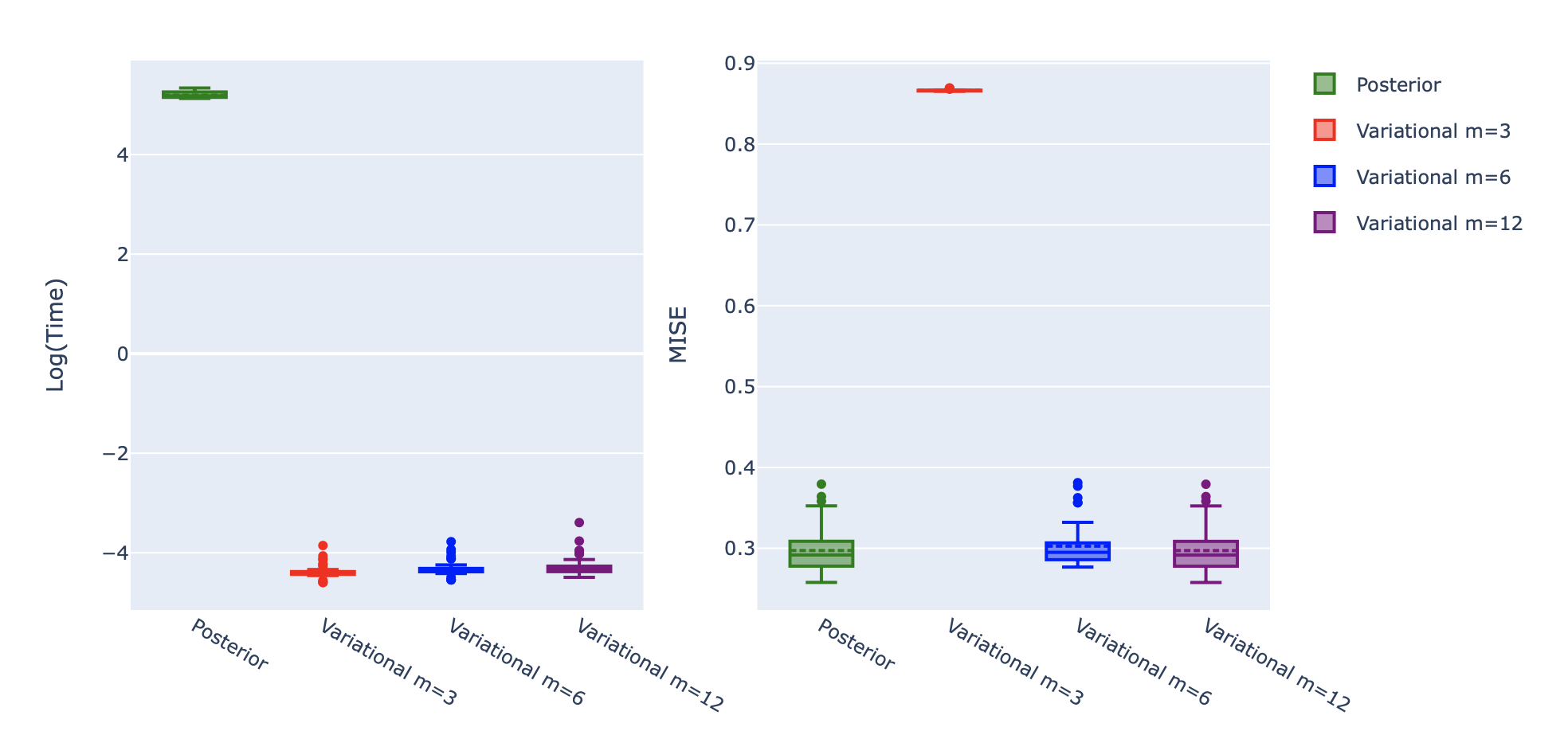}
  \caption{Boxplots of the (logarithm of) computation time (in seconds, on the left) and the MISE (on the right) of the true and variational posteriors for Gaussian series prior (sine basis) on the initial condition $\mu=f_0$ of the heat equation \eqref{eq: heat equation}, with $m = 3,6,12$ inducing variables from method \eqref{eq: inducing var op},  obtained from $50$ experiments with $n=4000$ samples.}
  \label{fig: box heat}
\end{figure}

\begin{figure}[h!]
\center
  \includegraphics[width=\textwidth]{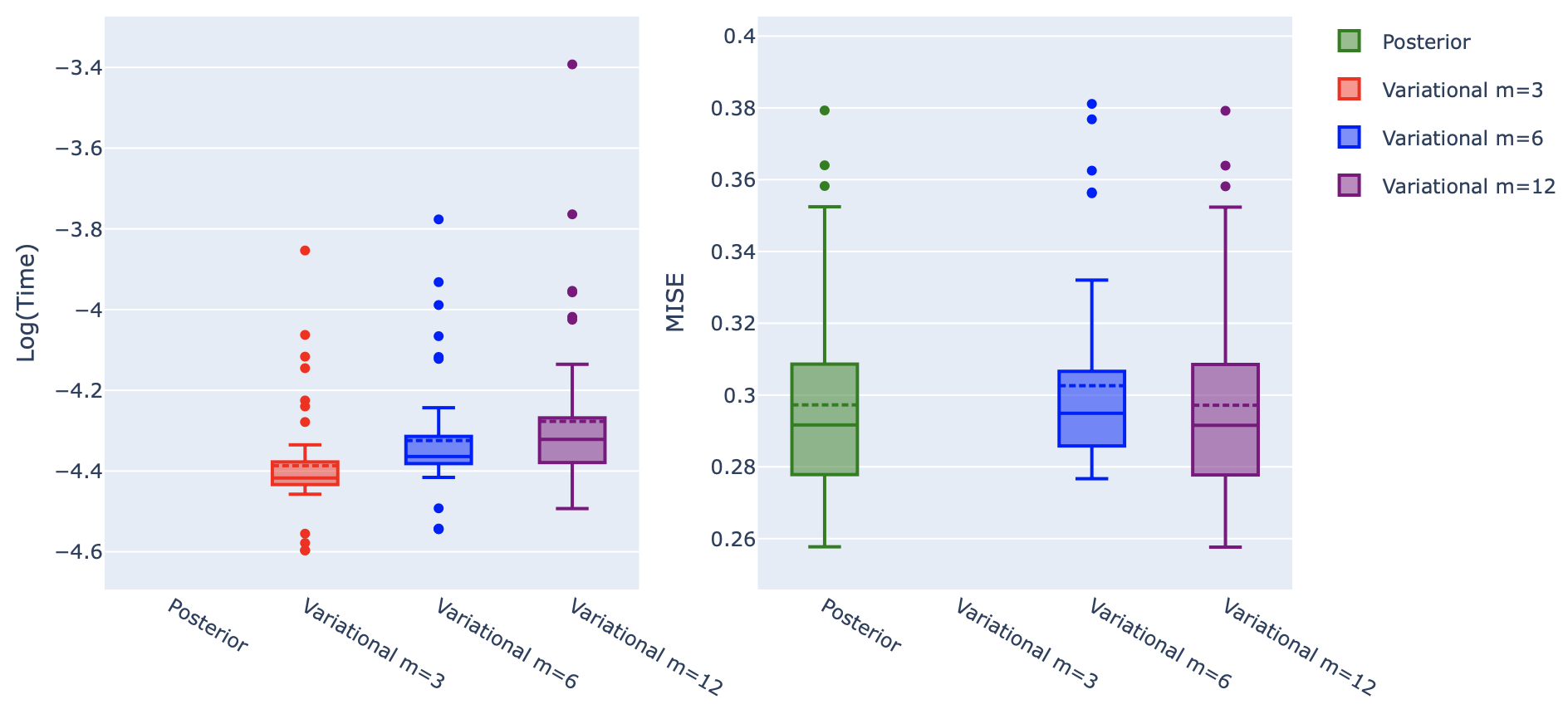}
  \caption{Zoom of Figure \ref{fig: box heat}.}
  \label{fig: box heat zoom}
\end{figure}

\subsection{Volterra operator}

Next we consider the Volterra operator \eqref{eq: volterra operator}. This is a mildly ill-posed problem of degree $p=1$. We set the sample size $n=15000$, take uniformly distributed covariates on $[0,1]$, and let 
\[ f_0(t) =\sqrt{2}  \sum\nolimits_j c_j j^{-(1+\beta)}\cos((j-1/2) \pi t),\quad c_j = \begin{cases} 1+0.9 \sin(\sqrt{3}\pi j), & j \text{ odd,}  \\ 
1+0.8 \sin(\sqrt{7}\pi j), & j \text{ even,} \end{cases}\]
for $\beta=0.6$, so that $f_0\in \bar{H}^\beta$. The independent observations are then generated as $Y_i\sim \mathcal{N}(\mathcal{A}f_0(x_i), 1)$, depending on the primitive of $f_0$.

We consider the prior with $\lambda_j = j^{-1-2\beta}$. In view of Corollary \ref{cor:heat} the optimal number of inducing variables is $m=n^{\frac{1}{3+2\beta}} \approx 10$. We consider the population spectral features method described in \eqref{eq: inducing var op} and plot the variational approximation of the posterior for $m=10$ and $m=5$ inducing variables in Figure \ref{fig: signal volterra}, using the same colorcode as in the previous section.

With $m=10$ on the top of Figure \ref{fig: signal volterra}, the variational approximation results in similar 95\% pointwise credible bands and posterior mean as the true posterior, providing an accurate approximation. On the bottom of this figure, we observe that the mean and pointwise credible bands with $m=5$ inducing variables are considerably different, though in both cases, the credible bands contain $f_0$.

\begin{figure}[h!]
\center
  \includegraphics[width=\textwidth]{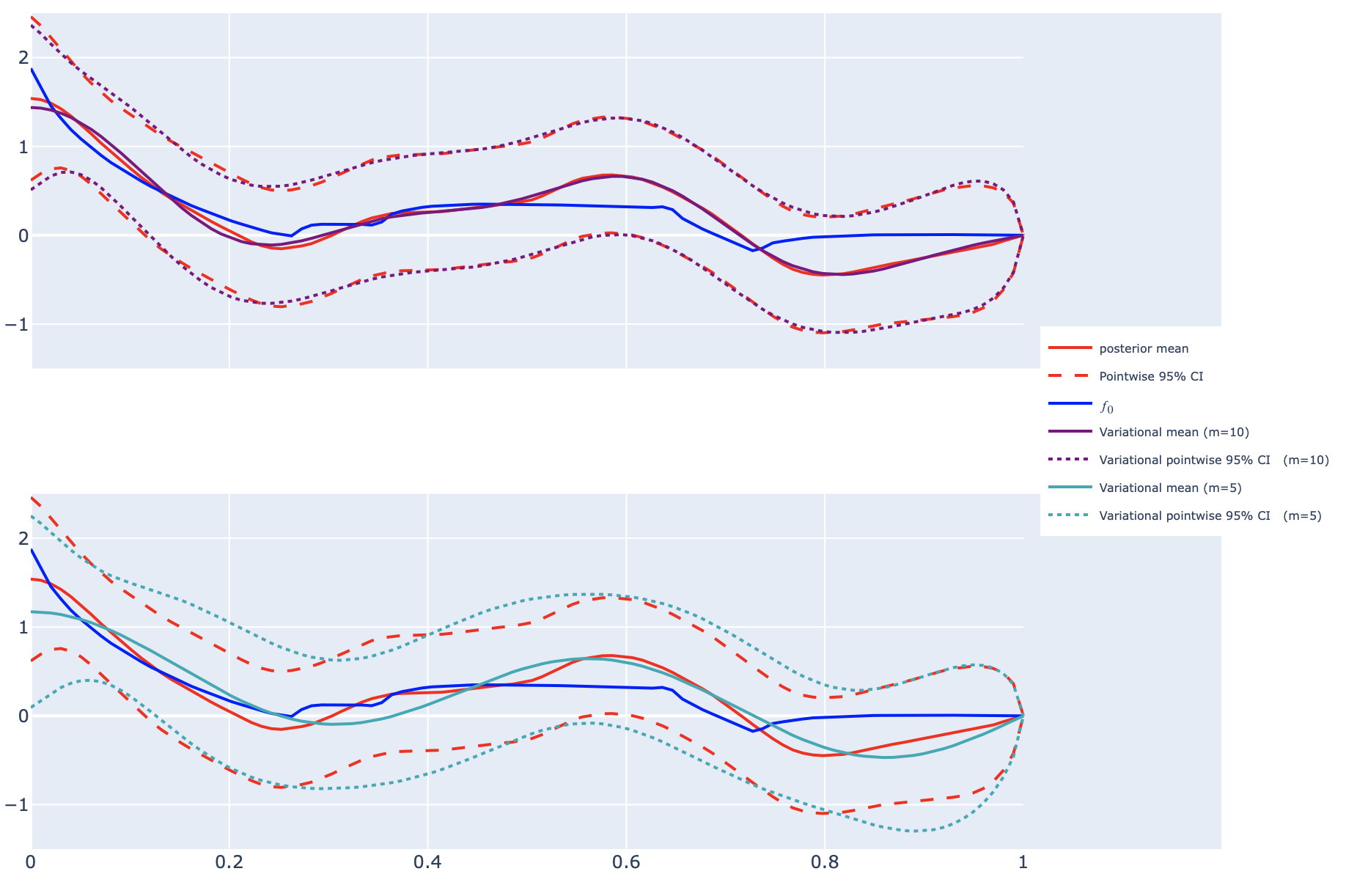}
  \caption{True and variational posterior means and credible regions for Gaussian series prior (cosine basis) and $m = 10$ (top) or $m=5$ (bottom) inducing variables from method \eqref{eq: inducing var op} for the Volterra operator in Section \ref{subsec: volterra}.}
  \label{fig: signal volterra}
\end{figure}

We repeat the above experiment $30$ times with $n=4000$ and compare the computation times and MISE \eqref{eq: MISE} of the true posterior and the variational posteriors obtained with the optimal choice of $m=8$ inducing variables and twice more/less variables $m=16$/$m=4$. Looking at Figures \ref{fig: box volterra} and \ref{fig: box volterra zoom}, the same message holds as before, in case of the heat equation.

\begin{figure}[h!]
\center
  \includegraphics[width=\textwidth]{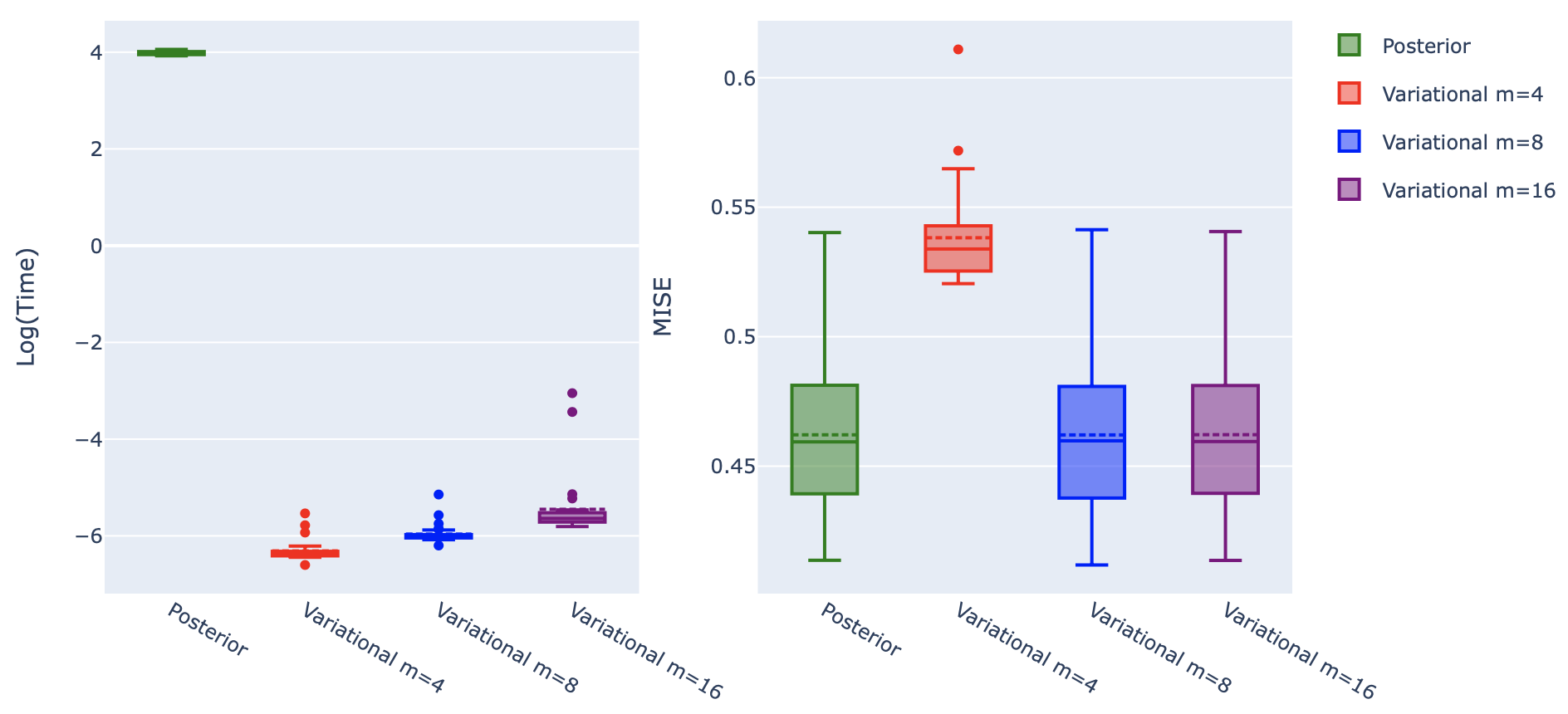}
  \caption{Boxplots of the (logarithm of) computation time (in seconds, on the left) and the MISE (on the right) of the true and variational posteriors for Gaussian series prior (cosine basis) with $m = 4,8,16$ inducing variables from method \eqref{eq: inducing var op},  obtained from $50$ experiments with $n=4000$ samples, for the Volterra operator in Section \ref{subsec: volterra}.}
  \label{fig: box volterra}
\end{figure}

\begin{figure}[h!]
\center
  \includegraphics[width=\textwidth]{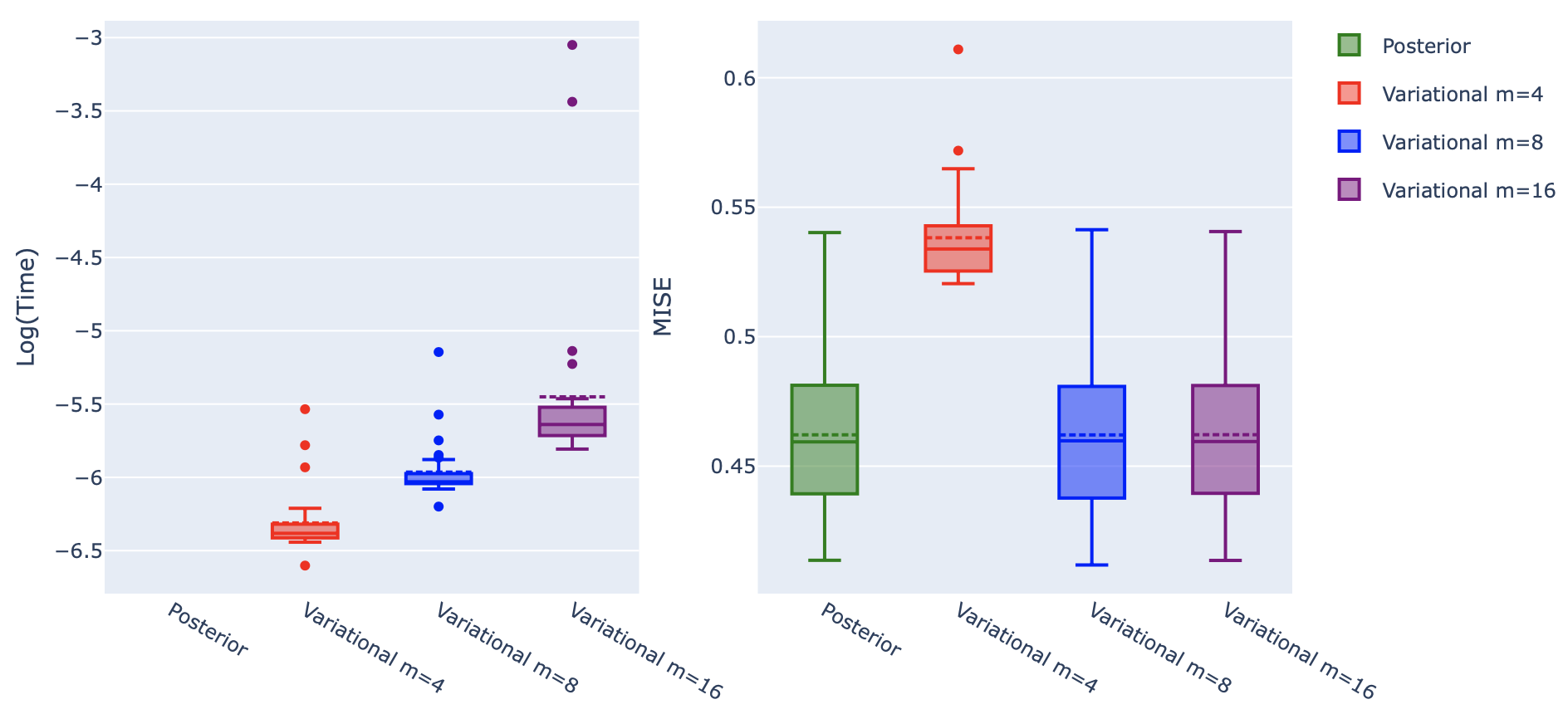}
  \caption{Zoom of Figure \ref{fig: box volterra}.}
  \label{fig: box volterra zoom}
\end{figure}

We also illustrate and compare theoretical and empirical phase-transition curves on synthetic data coming from the Volterra operator \eqref{eq: volterra operator}. We computed the (logarithm of the) ratio of the mean integrated squared error (MISE) corresponding to the true and variational posteriors (we simulate $20$ experiments each time to empirically approximate these quantities). We have considered $n$ ranging from $100$ to $10000$ and $m$ from $1$ to $17$, under the same setting as above. We have also plotted the phase transition curve (white line) coming from our theoretical analysis on Figure \ref{fig: phase transition}. One can note that the theoretical curve closely resembles the curve where the phase transition occurs in the empirical study. Indeed, there is not much empirical improvement of the MISE after the threshold given by Corollary \ref{cor: volterra}.

\begin{figure}[h!]
\center
  \includegraphics[scale=0.3]{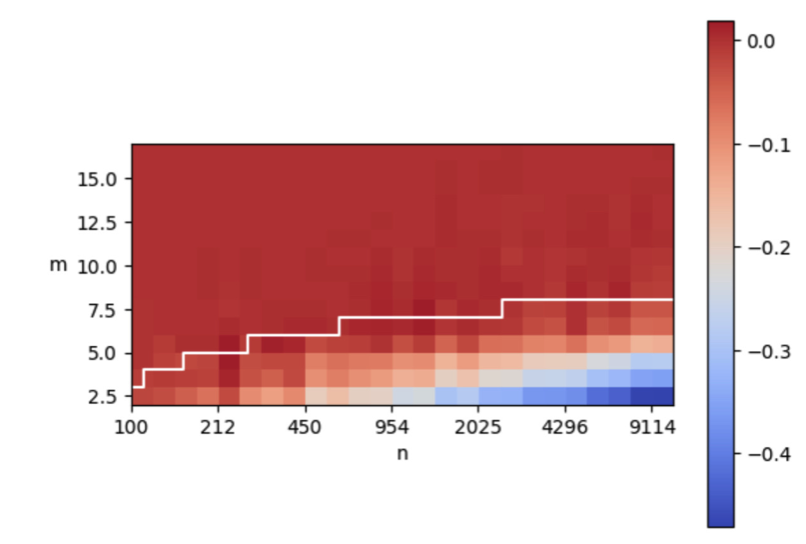}
  \caption{Log ratio $\log \frac{E_{f_0} \text{MISE}(\Pi[\cdot|X,Y])}{E_{f_0} \text{MISE}(\Psi^*)}$ of MISE between the true and variational posteriors recovering $f_0$ from its image by the Volterra operator in Section \ref{subsec: volterra}. In white is represented the function $n \to \lceil n^{\frac{1}{3+2\beta}} \rceil$ given by Corollary \ref{cor: volterra}.}
  \label{fig: phase transition}
\end{figure}

\subsection{Radon transform}

\subsubsection{Experiments}

We now turn to a simulation study of the Radon transform \eqref{eq: Radon transform}, which represents a mildly ill-posed problem of degree $p=1/4$. We observe the performance of the true posterior and variational approximations for different sample sizes,  $n=500$ and $n=5000$, and take independent covariates drawn from the distribution $dG(s,\phi)=2\pi^{-1}\sqrt{1-s^2}dsd\phi$ on $S=[0,1]\times [0,2\pi)$. We set the polar coordinates of the functional parameter $f_0$ on the unit disc $D=\left\{x\in \mathbb{R}^2:\ \norm{x}_2\leq 1\right\}$ as
\[ f_0(r,\theta) = \sum\nolimits_j c_j j^{-(1+\beta)}e_j(r,\theta),\quad c_j = \begin{cases} 1+0.5 \sin(\sqrt{3}\pi j), & j \text{ odd,}  \\ 
2+0.8 \sin(\sqrt{7}\pi j), & j \text{ even,} \end{cases}\]
for $\beta=0.6$ and 
\[ e_j(r,\theta) = 
\begin{cases} 
Z_{m_j}^{|l_j|}(r) \cos(l_j\theta), & l_j >0, \\ 
Z_{m_j}^{|l_j|}(r), & l_j=0,\\
Z_{m_j}^{|l_j|}(r) \sin(l_j\theta), & l_j<0,
 \end{cases}
\]
$m_j = \lceil\frac{\sqrt{1+8j}-1}{2}\rceil-1$ and $l_j = 2(j-1)-m_j(m_j+2)$. Note that $f_0\in \bar{H}^\beta$.  The independent observations are again generated as $Y_i\sim \mathcal{N}(\mathcal{A}f_0(x_i), 1)$.
 
We again consider the prior eigenvalues $\lambda_j = j^{-1-2\beta}$. The optimal number of inducing variables is $m=n^{\frac{2}{3+4\beta}}$, which for the different sample sizes $n=500$ and $n=5000$ we consider is equal to $10$ and $24$, respectively. We consider the population spectral feature method described in \eqref{eq: inducing var op} and plot the variational approximation of the posterior for $m$ and $\lceil m/4\rceil$ inducing variables in Figure \ref{fig: Radon_500} ($n=500$) and Figure \ref{fig: Radon_5000} ($n=5000$). For each setting, we represent the true posterior mean, the variational means, the upper and lower pointwise $2.5\%$ quantiles as well as the absolute pointwise difference between $f_0$ and the posterior/variational means $\hat{f}_n$. Negative values are represented in blue, positive ones in red and points corresponding to small absolute values are in white.

Again, similar conclusions can be drawn as in the previous sections. Observe that the true posterior and variational means are similar for the optimal choice of $m$, while choosing four times less inducing variables results in a posterior mean that is much smoother. On the other hand, the credible bands with the suboptimal choice of $m$ are overly large compared to the true posteriors.

\begin{figure}[h!]
\center
  \includegraphics[width=\textwidth]{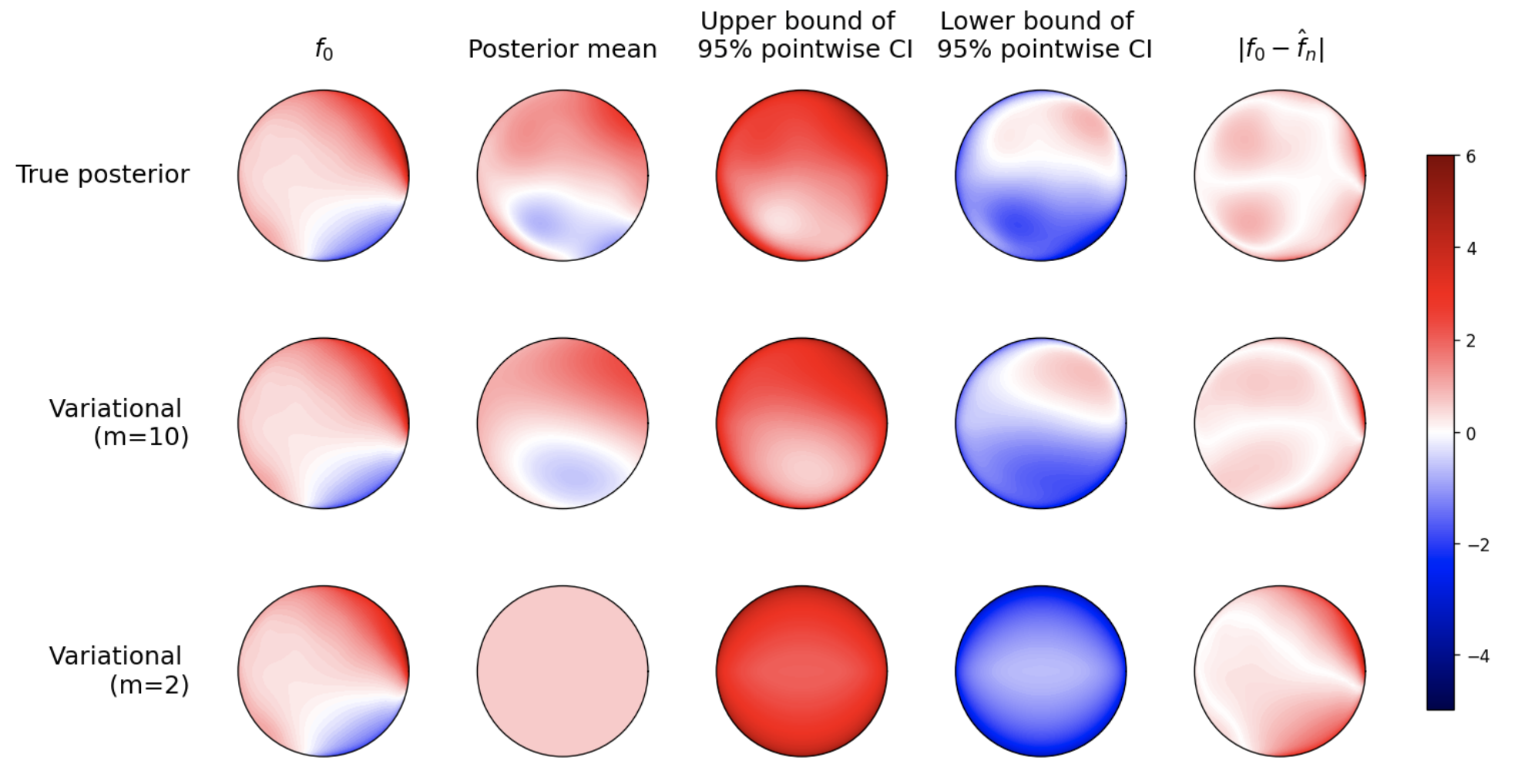}
  \caption{True and variational posterior means and credible regions for Gaussian series prior (Zernike polynomial basis) and $m = 17$ (middle) or $m=8$ (bottom) inducing variables from method \eqref{eq: inducing var op}, with $n=5000$, recovering the parameter $f_0$ from its Radon transforn (see Section \ref{subsec: radon}).}
  \label{fig: Radon_500}
\end{figure}

\begin{figure}[h!]
\center
  \includegraphics[width=\textwidth]{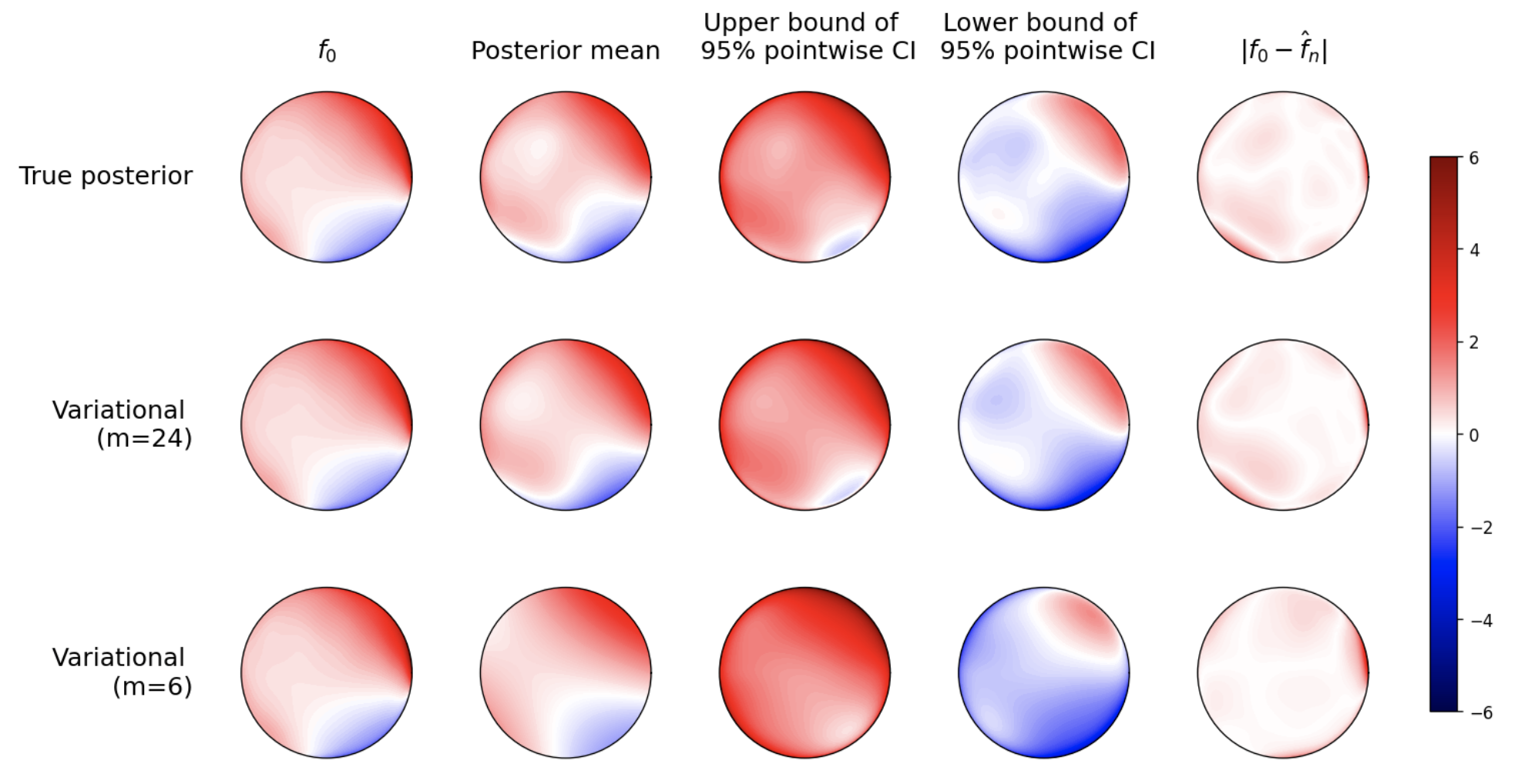}
  \caption{True and variational posterior means and credible regions for Gaussian series prior (Zernike polynomial basis) and $m = 24$ (middle) or $m=12$ (bottom) inducing variables from method \eqref{eq: inducing var op}, with $n=5000$, recovering the parameter $f_0$ from its Radon transforn (see Section \ref{subsec: radon}).}
  \label{fig: Radon_5000}
\end{figure}


\subsubsection{Applications}

The Radon transform is a mathematical technique with various applications, particularly in the field of medical imaging and image analysis. It is used to analyze and transform data from the spatial domain to the Radon domain, providing a different perspective on the data that can be useful for specific tasks. Inverting the Radon transform has found a lot of applications where lower-dimensional integrals of the inside of an object are more readily available than the object of interest itself.  We provide below a non-exhaustive list of possible applications:

\begin{itemize}
\item \textit{Computed Tomography (CT) Imaging and Medical Single Photon Emission Computed Tomography (SPECT)}: In CT scans, X-ray measurements are taken from different angles around a patient, and the Radon transform is used to reconstruct a cross-sectional image (slice) of the patient's body. This helps doctors visualize internal structures and diagnose various medical conditions. SPECT is a nuclear medicine imaging technique that uses gamma-ray detectors to generate 3D images of the distribution of radioactive tracers within a patient's body. The Radon transform is used in the image reconstruction process for SPECT. \cite{Ambartsoumian_2020, BARRETT1984, Natterer2001, ramlau2019}

\item  \textit{Seismic Imaging}: In seismology, the Radon transform is employed to process seismic data collected from earthquakes or controlled explosions. It helps create images of the subsurface structure of the Earth, aiding in oil and gas exploration and understanding geological formations. \cite{deHoop_2009, Raluca2016}

\item  \textit{Geophysical Imaging}: The Radon transform has applications in geophysical imaging techniques such as ground-penetrating radar (GPR), where it helps in image reconstruction to understand subsurface properties. \cite{2006Moysey}    

\item  \textit{Radar and Sonar Imaging}: The Radon transform is used in underwater sonar imaging to reconstruct images of underwater objects or terrains, or radar imaging to create imaging of landscapes. This has applications in marine biology, naval operations, and underwater exploration.\cite{Quinto_2011, Redding2004}

\item  \textit{Particle Tracking}: In high-energy physics and particle physics experiments, the Radon transform is used to analyze data from particle detectors to track the paths of particles, determining their trajectories and energies.\cite{RMankel_2004}

\item  \textit{Material Science and Crystallography}: The Radon transform can be applied to analyze diffraction patterns in crystallography and material science, helping to understand the structure of materials at the atomic level.\cite{Bernstein2013}
\end{itemize}

In Mathematics, the Radon transform has also been used to solve hyperbolic partial differential equations via the method of plane waves, which reduces the problem to the resolution of ordinary differential equations \cite{Helgason1999, Rim2017}.

\newpage
\section{Proof of Theorem 1}

We start by introducing some notation and background information used throughout the proof. First note, that since the eigenfunctions of the covariance kernel $k$ were chosen to be the eigenfunctions of $\mathcal{A}^*\mathcal{A}$, the prior $\Pi_\mathcal{A}$ on $\mathcal{A}f$, induced by the GP prior $\Pi$ on $f$, is also a centered Gaussian process with covariance kernel
\begin{align}
(x,y)\mapsto \sum_{j=1}^{\infty}\lambda_j\kappa_j^2 g_j(x)g_j(y),\label{def:prior:inv}
\end{align}
i.e. the eigenvalues and eigenfunctions of the kernel are $(\lambda_j\kappa_j^2)_{j\in\mathbb{N}}$ and $(g_j)_{j\in\mathbb{N}}$, respectively. Let us denote by  $\mathbb{H}_{\mathcal{A}}$ the corresponding  Reproducing Kernel Hilbert Space (RKHS) and by $\mathbb{H}$ the RKHS corresponding to the prior $\Pi$ on $f$. In view of Theorem I.18 of \cite{fundamentalsBNP}, the above RKHS takes the form
\begin{align}\label{def:rkhs:inverse}
 \mathbb{H}_{\mathcal{A}}=\left\{h(x)=\sum_{j=1}^\infty h_j g_j(x)\colon\ \norm{h}^2_{\mathbb{H}_{\mathcal{A}}}\coloneqq \sum_{j=1}^\infty h_j^2 \lambda_j^{-1}\kappa_{j}^{-2}<\infty \right\},
\end{align}
where $h_j=\langle h, g_j\rangle_{L_2(\mathcal{X};G)}$. Furthermore, note that for all measurable set $S\subset L_2(\mathcal{X};G)$ we have for $Z_j\sim^{iid} N(0,1)$ random variables, that
\begin{align}
\Pi(f:\, \mathcal{A}f\in S)=P(\sum_{j=1}^\infty \lambda_j^{1/2}\kappa_j Z_j g_j\in S)=\Pi_{\mathcal{A}}(w:\, w\in S).\label{eq:help:prob:Gaussians}
\end{align}
In the next sections, we denote the rates for the direct problems by
\begin{equation}\label{eq:def:eps_n}
\varepsilon_n=M
\begin{cases}
n^{-\frac{\alpha\wedge\beta+p}{1+2\alpha+2p}}& \text{ in the mildly ill-posed case,}\\
n^{-c/(\xi+2c)} \log^{-\beta/p+c\alpha/(\xi+2c)} (n)& \text{ in the severely ill-posed case,}
\end{cases}
\end{equation}
 for some $M>0$ large enough.

Finally, we provide an explicit formula of the KL divergence between the posterior distribution $\Pi[\cdot\ |X,Y]$ and the variational approximation $\Psi^*$. It can be expressed with the evidence lower bound $\mathcal{L}$ as
\[\text{KL}\left(\Psi^*|\! |\Pi[\cdot | X,Y]\right) = \log p(X,Y) - \mathcal{L},\]
where computations from \cite{titsias2009variational} give
\begin{align*}
\mathcal{L} &\coloneqq \log \int \exp\Big(\int \log p_f(X,Y) d\Pi(f|\mathbf{u})\Big) d\Pi_u(\mathbf{u})\\
&= -\big|2\pi\big(\sigma^2I_n+K_{\boldsymbol{\mathcal{A}f}\boldsymbol{\mathcal{A}f}}\big)\big|-\frac{1}{2\sigma^2}\mathbf{y}\left(\sigma^2I_n+K_{\boldsymbol{\mathcal{A}f}\boldsymbol{\mathcal{A}f}}\right)^{-1}\mathbf{y}-Tr\left(K_{\boldsymbol{\mathcal{A}f}\boldsymbol{\mathcal{A}f}}-Q_{\boldsymbol{\mathcal{A}f}\boldsymbol{\mathcal{A}f}}\right),
\end{align*}
for  $Q_{\boldsymbol{\mathcal{A}f}\boldsymbol{\mathcal{A}f}}=K_{\boldsymbol{\mathcal{A}f}\boldsymbol{u}}K_{\boldsymbol{u}\boldsymbol{u}}^{-1}K_{\boldsymbol{u}\boldsymbol{\mathcal{A}f}}$. Then the KL divergence takes the form
\begin{equation}\label{eq: KL post var}
\begin{split}
 \text{KL}\left(\Psi^*|\! |\Pi[\cdot | X,Y]\right) &= \frac{1}{2}\Big(\mathbf{y}\big[\big(\sigma^2I_n+Q_{\boldsymbol{\mathcal{A}f}\boldsymbol{\mathcal{A}f}} \big)^{-1}-\big(\sigma^2I_n+K_{\boldsymbol{\mathcal{A}f}\boldsymbol{\mathcal{A}f}} \big)^{-1}\big]\mathbf{y}^T\\
&\qquad +\log\frac{|\sigma^2I_n+Q_{\boldsymbol{\mathcal{A}f}\boldsymbol{\mathcal{A}f}}|}{|\sigma^2I_n+K_{\boldsymbol{\mathcal{A}f}\boldsymbol{\mathcal{A}f}}|}+\frac{1}{\sigma^2}Tr(K_{\boldsymbol{\mathcal{A}f}\boldsymbol{\mathcal{A}f}}-Q_{\boldsymbol{\mathcal{A}f}\boldsymbol{\mathcal{A}f}})\Big).
\end{split}
\end{equation}

\subsection{Step 1: Empirical $L_2$ contraction in the direct problem}
As a first step we fix the design points and derive posterior contraction rate around $\mathcal{A}f_0$ with respect to the empirical $L_2(\mathcal{X};P_n)$-norm, i.e. $\|w\|_{L_2\left(\mathcal{X}; P_n\right)}^2=n^{-1}\sum_{i=1}^n w(x_i)^2$. More precisely, we show that there exists an event $A_n$ with $P_X(A_n)\to 1$ and events $B_{n,|X}$ conditional on the design $X$ with  $\inf_{X\in A_n}P_{Y\given X}(B_{n,|X})\to 1$, such that for any sequence $M_n\rightarrow \infty$ and $X\in A_n$
 \begin{equation}\label{eq: post rate empirical l2 norm} 
 E_{Y\given X}\Pi\left[f\colon\|\mathcal{A}f-\mathcal{A}f_0\|_{L_2\left(\mathcal{X}; P_n\right)} \geq M_n\varepsilon_n\given\ X, Y\right]\mathds{1}_{B_{n,|X}}\leq Ce^{-cM_n^2n\varepsilon_n^2}
 \end{equation}
holds for $\varepsilon_n$ given in \eqref{eq:def:eps_n}.

Let us recall the definition of the concentration function (in case of the direct problem)
 \[ \phi_{\mathcal{A}f_0}(\varepsilon) \coloneqq \underset{h\in \mathbb{H}_{\mathcal{A}}: \|\mathcal{A}f_0-h\|_{L_2\left(\mathcal{X}; P_n\right)}< \varepsilon}{\inf}\ \norm{h}^2_{\mathbb{H}_{\mathcal{A}}}-\log \Pi_{\mathcal{A}}\left(w:\,\|w\|_{L_2\left(\mathcal{X};P_n\right)}<\varepsilon\right).\]
Then in view of  Theorem 3.3 of \cite{MR2418663}, to prove \eqref{eq: post rate empirical l2 norm} it is sufficient to verify the concentration inequality
\begin{equation}\label{eq: concentration inequality} 
 \phi_{\mathcal{A}f_0}(\varepsilon_n)\leq n\varepsilon_n^2.
 \end{equation}
This result is based on \cite{MR2332274} where in the proof it is shown that there exists a sequence of events $B_{n,|X}$ such that $\sup_{X}P_{Y\given X}\left(B_{n,|X}^c\right)$ vanishes and \eqref{eq: post rate empirical l2 norm} holds $P_X$-almost surely.

We prove \eqref{eq: concentration inequality} in two steps. First we verify it for the $L_2(\mathcal{X};G)$-norm, i.e. we show that for $M$ large enough in \eqref{eq:def:eps_n},
\begin{align}
\underset{h\in\mathbb{H}_{\mathcal{A}}: \norm{h-\mathcal{A}f_0}_{L_2(\mathcal{X};G)}\leq \varepsilon_n}{\inf} \norm{h}^2_{\mathbb{H}_{\mathcal{A}}} & \leq n\varepsilon_n^2,\label{eq: approx rkhs eq}\\
-\log \Pi_\mathcal{A}\left(w:\, \|w\|_{L_2(\mathcal{X};G)}<\varepsilon_n\right) &\leq n\varepsilon_n^2.\label{eq: small ball 0}
\end{align}
Then we relate the population $L_2(\mathcal{X};G)$-norm to the empirical $L_2\left(\mathcal{X}; P_n\right)$-norm on a large enough event $A_n$, finishing up the argument. We note that one can not apply this result to the $L_2(\mathcal{X};G)$-norm as the testing metric (Hellinger) and the $L_2$-norm do not satisfy the required connection.

In the mildly ill-posed case the above inequalities directly follow from  Lemma \ref{lemma: RKSH approx} and \ref{lemma: small ball prob 0}, respectively. In the severely ill-posed case for \eqref{eq: approx rkhs eq} in view of Lemma \ref{lemma: RKSH approx} it is sufficient to verify that  $J_{\varepsilon_n}^{\alpha-2\beta}e^{\xi J_{\varepsilon_n}^p}\leq n\varepsilon_n^2$ for $J_{\varepsilon_n}^{\beta} e^{c J^p_{\varepsilon_n} }\asymp \varepsilon_n^{-1}$. Note that by substituting $\varepsilon_n$ in the previous inequality, we equivalently get $J_{\varepsilon_n}^{\alpha}e^{(\xi+2c) J_{\varepsilon_n}^p} \lesssim n$.
Then, in view of Section 3.3 of \cite{MR3757524} (using the Lambert function) this holds for some $J_{\varepsilon_n}= O(\log^{1/p} n)$. Furthermore, following from $e^{J_{\varepsilon_n}^p} \lesssim \left(nJ_{\varepsilon_n}^{-\alpha}\right)^{1/(\xi+2c)}$, we arrive at
\begin{align*}
\varepsilon_n\asymp J_{\varepsilon_n}^{-\beta} e^{-c J^p_{\varepsilon_n} } \gtrsim n^{-c/(\xi+2c)} \log^{-\beta/p+c\alpha/(\xi+2c)} (n),
\end{align*}
finishing the proof of \eqref{eq: approx rkhs eq}. For \eqref{eq: small ball 0}, in view of Lemma \ref{lemma: small ball prob 0}, we need  $\varepsilon_n\gtrsim n^{-1/2} \log^{(p+1)/2p} n$, which holds for $\varepsilon_n$.

It remained to replace in \eqref{eq: approx rkhs eq} and \eqref{eq: small ball 0} the  $L_2\left(\mathcal{X}; G\right)$-norm with the $L_2\left(\mathcal{X}; P_n\right)$-norm. First note that in view of Lemma \ref{lem:norms}, there exists an event $A_{n,1}$ with $P_X(A_{n,1}^c)=o(1)$ such that for $X\in A_{n,1}$
\[
 \Pi_\mathcal{A}\big(\norm{w}_{L_2(\mathcal{X};P_n)}< C\varepsilon_n\big)\geq \Pi_\mathcal{A}\big(\norm{w}_{L_2(\mathcal{X};G)}<\varepsilon_n\big) + o\big(e^{-n\varepsilon_n^2}\big) \gtrsim e^{-n\varepsilon_n^2}.\]
Furthermore, note that the upper bound in Lemma \ref{lemma: RKSH approx} were derived for $h=\mathcal{A}f_0^{J_{\varepsilon}}$ with appropriately chosen $J_\varepsilon$. Then in view of Lemma \ref{lem:Af_0} (with $J=J_{\varepsilon_n}< k$ in the lemma) there exists an event $A_{n,2}$ with $P_X(A_{n,2}^c)=o(1)$ such that
\begin{align*}
\|h-\mathcal{A}f_0\|_{L_2\left(\mathcal{X}; P_n\right)}=\|\mathcal{A} f_0^{\perp J_{\varepsilon_n}} \|_{L_2\left(\mathcal{X}; P_n\right)} \lesssim \|\mathcal{A} f_0^{\perp J_{\varepsilon_n}}\|_{L_2\left(\mathcal{X}; G\right)}+o(\varepsilon_n)\lesssim \varepsilon_n,
\end{align*}
where $\mathcal{A} f_0^{\perp J}(x)=\sum_{j=J+1}^\infty \kappa_j f_{0,j}g_j(x)$ and we used that $(\alpha \wedge \beta)+p > 3/2 + 2\gamma$ in the first bound, verifying the statement on the event $A_n=A_{n,1}\cap A_{n,2}$ satisfying $P_X(A_n^c)=o(1)$, for some large $M>0$.

\subsection{Step 2: Population $L_2$ contraction rate in the direct problem}
Next we adapt the contraction rate result \eqref{eq: post rate empirical l2 norm} to the random design regression model and consider $L_2(\mathcal{X};G)$ contraction rate, i.e. we show that there exists a sequence of events $D_n$ with $P_{X,Y}(D_n^c)=o(1)$ such that
\begin{equation}\label{eq:population} 
 E_{f_0}\Pi\left[f\colon\|\mathcal{A}f-\mathcal{A}f_0\|_{L_2\left(\mathcal{X}; G\right)} \geq M_n\varepsilon_n\given\ X, Y\right]\mathds{1}_{D_n}\leq Ce^{-cM_n^2n\varepsilon_n^2}.
 \end{equation}

First note that in view of Lemma \ref{lem:norms} for $f\in\mathcal{F}_n$ defined in \eqref{def: sieve} we have on an event $A_{n,1}$ with $P_X(A_{n,1}^c)=o(1)$ that $\|\mathcal{A}f-\mathcal{A}f_0\|_{L_2\left(\mathcal{X}; G\right)}\leq C(\|\mathcal{A}f-\mathcal{A}f_0\|_{L_2\left(\mathcal{X}; P_n\right)}+\varepsilon_n)$. Furthermore, note that \eqref{eq: approx rkhs eq} and \eqref{eq: small ball 0} in view of Proposition 11.19 of \cite{fundamentalsBNP} imply that for some $c>0$
\begin{align}
\Pi[f\colon\ \norm{\mathcal{A}f-\mathcal{A}f_0}_{L_2\left(\mathcal{X};G\right)}\leq \epsilon_n ]\gtrsim e^{-c n\varepsilon_n^2}.\label{eq:decentered:small:ball}
\end{align}
In view of $\Pi(f\in \mathcal{F}_n^c)\leq e^{-n^{\frac{1}{1+2\gamma}} n\varepsilon_n^2}$, see Lemma \ref{lem:norms}, Lemma \ref{lemma: post mass events} gives $\Pi(f\in \mathcal{F}_n^c|X,Y)\leq e^{-n^{\frac{1}{1+2\gamma}}n\varepsilon_n^2/2}$. Furthermore, in view of \eqref{eq: post rate empirical l2 norm} there exists an event $A_{n,2}$ with $P_{X,Y}(A_{n,2}^c)=o(1)$ such that
\[E_X E_{Y|X}\Pi\left[f \colon\|\mathcal{A}f-\mathcal{A}f_0\|_{L_2\left(\mathcal{X}; P_n\right)} \geq M_n\varepsilon_n\given\ X, Y\right]\mathds{1}_{A_{n,2}}\lesssim e^{-cM_n^2n\varepsilon_n^2}.\]
Therefore, by taking $D_n=A_{n,1}\cap A_{n,2}$ we get that
\begin{align*}
E_{f_0}\Pi&\left[f\colon\|\mathcal{A}f-\mathcal{A}f_0\|_{L_2\left(\mathcal{X}; G\right)} \geq M_n\varepsilon_n\given X, Y\right]\mathds{1}_{A_n}\\
&\leq E_{f_0}\Pi\left[f\in \mathcal{F}_n\colon\|\mathcal{A}f-\mathcal{A}f_0\|_{L_2\left(\mathcal{X}; G\right)} \geq M_n\varepsilon_n\given X, Y\right]\mathds{1}_{A_n}+ E_{f_0}\Pi\left[f\in \mathcal{F}_n^c\given X, Y\right]\\
&\leq E_X \Big( E_{Y|X}\Pi\left[f \colon\|\mathcal{A}f-\mathcal{A}f_0\|_{L_2\left(\mathcal{X}; P_n\right)} \geq CM_n\varepsilon_n\given X, Y\right]\mathds{1}_{A_n}\Big)+e^{-n^{\frac{1}{1+2\gamma}}n\varepsilon_n^2/2}\\
&\lesssim e^{-\left(n^{\frac{1}{1+2\gamma}}\ \wedge\ M_n^2\right)n\varepsilon_n^2/2}.
\end{align*}

\subsection{Step 3: Population $L_2$ contraction rate in the indirect problem}
Next, we turn the contraction rate results for $\mathcal{A}f$ in the direct problem to contraction rates in the indirect problem for $f$. We show that there exists an event $A_n$ with $P_{X,Y}\left(A_n\right)\to 1$, such that for any $M_n\rightarrow\infty$
 \begin{equation}\label{eq: true post rate indirect problem}
  E_{f_0}\Pi\left[f\colon\ \norm{f-f_0}_{L_2(T; \mu)} \geq M_n \varepsilon_n^{\text{inv}}\given\ X\right]\mathds{1}_{A_n}\leq Ce^{-c\left(n^{\frac{1}{1+2\gamma}}\ \wedge\ M_n^2\right) n \varepsilon_n^2}.
  \end{equation}

The proof follows the lines of Lemma 2.1 of \cite{MR3757524}. Let us define
\[  \mathcal{S}_n\coloneqq \Big\{f\in L_2(\mathcal{T};\mu):\ \sum_{j> k_n} \langle f, e_j\rangle^2 \leq r\rho_n^2\Big\},\]
where the parameters $k_n$, $r>0$ and $\rho$ will be specified later, depending on the degree of ill-posedness. Then let us define the modulus of continuity as
\begin{equation}\label{eq: modulus continuity} 
\delta_n = \sup \left\{\norm{f-f_0}_{L_2(\mathcal{T}; \mu)}:\ f\in \mathcal{S}_n,\, \|\mathcal{A}f-\mathcal{A}f_0\|_{L_2(\mathcal{X}; G)}\leq M_n \varepsilon_n\right\}
\end{equation}
and note that in view of (3.4) from \cite{knapik:salomond:2018}
\begin{equation}\label{eq: rate modulus inv} 
\delta_n \lesssim M_n\kappa_{k_n}^{-1}\varepsilon_n + \rho_n + k_n^{-\beta}.
\end{equation}
Furthermore, the definition of $\delta_n$ implies that
\begin{align*}
&E_{f_0}\Pi\left[f\colon\ \norm{f-f_0}_{L_2(\mathcal{T}; \mu)} \geq \delta_n \given X,Y\right]\mathds{1}_{A_n}\\
&\qquad\leq E_{f_0}\Pi\left[f\in\mathcal{S}_n \colon\ \norm{f-f_0}_{L_2(\mathcal{T}; \mu)} \geq \delta_n \given X,Y\right]\mathds{1}_{A_n} + E_{f_0} \Pi\left[\mathcal{S}_n^c\ | X,Y\right]\mathds{1}_{A_n}\\
&\qquad\leq E_{f_0}\Pi\left[f\in\mathcal{S}_n \colon\ \norm{\mathcal{A}f-\mathcal{A}f_0}_{L_2(\mathcal{X}; G)} \geq M_n\varepsilon_n \given X,Y\right]\mathds{1}_{A_n} + E_{f_0} \Pi\left[\mathcal{S}_n^c\ | X,Y\right]\mathds{1}_{A_n}.
\end{align*}
In view of \eqref{eq:population} the first term on the right hand side tends to zero for any $A_n\subset D_n$. We show below both in the mildly and severely ill-posed inverse problems, that for appropriate choices of $k_n$, $\rho_n$ and $r>0$, we have $\delta_n\lesssim M_n \varepsilon_n^{inv}$ and the second term on the right hand side of the previous display tends to zero.

First we consider the mildly ill-posed problem and set
\[k_n=n^{\frac{1}{1+2\alpha+2p}},\quad \rho_n=M_nn^{-\frac{\alpha\wedge\beta}{1+2\alpha+2p}},\quad \varepsilon_n=n^{-\frac{\alpha\wedge\beta+p}{1+2\alpha+2p}}.\]
Then, in view of \eqref{eq: rate modulus inv}  we have  $\delta_n\lesssim M_n \varepsilon_n^{inv}$, hence it remains to show that
\begin{align}
E_{f_0} \Pi\left[\mathcal{S}_n^c\ | X,Y\right]\mathds{1}_{A_n}\lesssim e^{-cM_n^2n\varepsilon_n^2}.\label{UB:SN^c}
 \end{align}
Note that Lemma 5.2 of \cite{MR3757524} for $r>2(1+2\alpha)/\alpha$ (remarking that $\rho_n^2k_n^{1+2\alpha}=M_n^2n\varepsilon_n^2= :n\epsilon_n^2$) provides that  
\begin{equation}\label{eq: mass to control sn} \Pi\left[\mathcal{S}_n^c\right]\leq e^{-Cn(M_n\varepsilon_n)^2}.\end{equation}
This together with  \eqref{eq:decentered:small:ball} imply in view of Lemma \ref{lemma: post mass events}  (with $\epsilon_n = M_n\varepsilon_n$) the inequality \eqref{UB:SN^c}.

We now turn to the severely ill-posed case and set $k_n=J_{\varepsilon_n}=O(\log^{1/p} n)$ and $\rho_n=M_n\log^{-\beta/p} n$.
Since $J_{\varepsilon_n}^{\beta} e^{c J^p_{\varepsilon_n} }\asymp \varepsilon_n^{-1}$ it implies $\kappa_{k_n}^{-1}\varepsilon_n \lesssim e^{cJ_{\varepsilon_n}^p}\varepsilon_n\lesssim  J_{\varepsilon_n}^{\beta} \lesssim \log^{-\beta/p} n$, therefore, in view of the arguments  above it only remains to show \eqref{eq: mass to control sn}. We proceed as in the proof of Lemma 5.2 of \cite{MR3757524} and find that, for $Z_j\sim^{iid}N(0,1)$, whenever $t<\left(2\lambda_j\right)^{-1}$ for $j>k_n$, 
\begin{align*}
 \Pi\left[\mathcal{S}_n^c\right] &= P\Big(\sum_{j>k_n} \lambda_j Z_j^2 > r\rho_n^2\Big)\\
 &=  P\Big(\exp\Big(t\sum_{j>k_n} \lambda_j Z_j^2 \Big)> \exp\big(tr\rho_n^2\big)\Big)\\
 &\leq \exp\big(-tr\rho_n^2\big) E \exp\Big(t\sum_{j>k_n} \lambda_j Z_j^2 \Big)\\
 &= \exp\big(-tr\rho_n^2\big) \prod_{j>k_n}  E \exp\big(t\lambda_j Z_j^2 \big)\\
 &= \exp\big(-tr\rho_n^2\big) \prod_{j>k_n}   \big(1-2t\lambda_i\big)^{-1/2}.
\end{align*}
Since $\log(1-y)\geq -y/(1-y)$ for $y<1$,
\[
\log \Pi\left[\mathcal{S}_n^c\right]\leq -rt\rho_n^2 + \sum_{j>k_n} \frac{t\lambda_j}{1-2t\lambda_j}.
\]
Choosing $t=\lambda_{k_n}^{-1}/4$, the second term on the right-hand side above is upper-bounded by a constant. As
\begin{align*}
t\rho_n^2&\asymp M_n^2 \lambda_{k_n}^{-1} \log^{-2\beta/p} n
\asymp M_n^2k_n^{\alpha} e^{\xi k_n^p} \log^{-2\beta/p} n\\
&\asymp  M_n^2k_n^{\alpha} \left(nk_n^{-\alpha}\right)^{\xi/(\xi+2c)} \log^{-2\beta/p} n\\
&\asymp M_n^2n^{\xi/(\xi+2c)}\left(\log n\right)^{-\frac{2\beta}{p}+\frac{2c\alpha}{p(\xi + 2c)}} = M_n^2n\varepsilon_n^2,
\end{align*}
the result is proved with $r$ large enough.

\subsection{Step 4: Contraction rate for the VB posterior}

Finally, we replace the true posterior by the variational posterior $\Psi^*$ in \eqref{eq: true post rate indirect problem}. We can apply Lemma \ref{lemma: var rates transition} with $\Delta_n=n\left(M_n\varepsilon_n\right)^2$ so that, for $M_n\to \infty$,
 \begin{align*}
  E_{f_0}&\Psi^*\left[f\colon\ \norm{f-f_0}_{L_2(T; \mu)} \geq M_n\varepsilon_n^{\text{inv}}\right]\mathds{1}_{A_n}\\
  &\leq \frac{2}{\left(n^{\frac{1}{1+2\gamma}}\ \wedge\ M_n^2\right)n\varepsilon_n^2}\left(E_{f_0} KL(\Psi^*|\! |\Pi[\cdot\given X,Y])\mathds{1}_{A_n(X,Y)} + Ce^{-\left(n^{\frac{1}{1+2\gamma}}\ \wedge\ M_n^2\right)n\varepsilon_n^2/2}\right).
 \end{align*} 
   Since, $M_n\to\infty$ and $n\varepsilon_n^2\to\infty$, the conclusion then follows if $E_0 KL(\Psi^*|\! |\Pi[\cdot\given X,Y])\leq Cn\varepsilon_n^2$.
According to Lemma 3 in \cite{JMLR:v23:21-1128} and \eqref{eq: KL post var}, for any $h\in\mathbb{H}_{\mathcal{A}}$,
\begin{align*}
 E_{f_0}& KL(\Psi^*|\! |\Pi[\cdot\given X,Y]) \leq \sigma^{-2}\big(n\norm{\mathcal{A}f_0-h}^2_{L_2(\mathcal{X};G)}+\norm{h}^2_{\mathbb{H}_{\mathcal{A}}}E_x \norm{ K_{\boldsymbol{\mathcal{A}f}\boldsymbol{\mathcal{A}f}}-Q_{\boldsymbol{\mathcal{A}f}\boldsymbol{\mathcal{A}f}}}\\
 &\qquad+E_x Tr\left(K_{\boldsymbol{\mathcal{A}f}\boldsymbol{\mathcal{A}f}}-Q_{\boldsymbol{\mathcal{A}f}\boldsymbol{\mathcal{A}f}}\right) \big).
\end{align*}
Then in view of Lemma \ref{lemma: RKSH approx}, for $n$ large enough, there exists $h\in\mathbb{H}_{\mathcal{A}}$ such that $\norm{\mathcal{A}f_0-h}_{L_2(\mathcal{X};G)}\leq \varepsilon_n$ and $\norm{h}_{\mathbb{H}_{\mathcal{A}}}\leq n\varepsilon_n^2$. Hence the claimed upper bound follows from the assumptions on the trace and spectral norm term.

\subsection{Technical lemmas}

\begin{lemma}[RKHS approximation for random series priors]\label{lemma: RKSH approx}
Let $f_0\in \bar{H}^{\beta}$, $\beta>0$, and consider the centered GP prior $\Pi_{\mathcal{A}}$  on $\mathcal{A}f$ given in \eqref{def:prior:inv}. Then
\[\underset{h\in\mathbb{H}_{\mathcal{A}}: \norm{h-\mathcal{A}f_0}_{L_2(\mathcal{X};G)}\leq \epsilon}{\inf} \norm{h}^2_{\mathbb{H}_{\mathcal{A}}} \lesssim
\begin{cases} \epsilon^{-\frac{2\alpha-2\beta+1}{\beta+p}} &\mbox{if $\kappa_j\asymp j^{-p}, \lambda_j\asymp j^{-1-2\alpha}$ for $\alpha>0,p\geq0$, $\beta\leq 2\alpha+1$}\\
&\\
J_\epsilon^{\alpha-2\beta}e^{\xi J_\epsilon^p}& \mbox{if $\kappa_j\asymp e^{-c j^p},   \lambda_j\asymp j^{-\alpha}e^{-\xi j^p} $, for $\alpha \geq 0, \xi>0\text{ or }$}\\
&\mbox{$\xi=0, \alpha\geq 2\beta$,\text{ and $p\geq1$}},
\end{cases}\] 
where $J_\epsilon$ is the smallest integer such that ${\max}_{ j\geq J_\epsilon} (\kappa_j j^{-\beta})  \|f_0\|_{\beta}\leq \epsilon$.
\end{lemma}
\begin{proof}
For simplicity let us denote by $w=\mathcal{A}f_0$ and note that for any $J\in \N$, the function $w^J(x)=\sum_{j=1}^J w_j  g_j(x)\in \mathbb{H}_{\mathcal{A}}$, with $ w_j=\langle w,g_j\rangle_{L_2(\mathcal{X};G)}$. Then in view of \eqref{def:rkhs:inverse} and using the notation $f_{0,j}= \langle f_0,e_j\rangle_{L_2(\mathcal{T};\mu)}$,
\begin{align*}
\big\|{w^J}\big\|_{\mathbb{H}_{\mathcal{A}}}^2&= \sum_{j=1}^J \kappa_j^{-2}\lambda_j^{-1} w_j^2
=  \sum_{j=1}^J  j^{-2\beta}\lambda_j^{-1} f_{0,j}^2 j^{2\beta}
\leq {\max}_{1\leq j\leq J}  (j^{-2\beta}\lambda_j^{-1}) \|f_0\|_{\beta}^2,\\
\norm{w^J-w}_{L_2(\mathcal{X};G)}^2&=\sum_{j={J+1}}^\infty w_j^2=\sum_{j={J+1}}^\infty \kappa_j^2j^{-2\beta} f_{0,j}^2j^{2\beta}\leq {\max}_{ j\geq J} (\kappa_j^2j^{-2\beta})  \|f_0\|_{\beta}^2.
\end{align*}

Then, in the mildly ill-posed inverse problem (with $\kappa_j\asymp j^{-p}, \lambda_j\asymp j^{-1-2\alpha}$), the smallest $J_\epsilon\in\mathbb{N}$ such that ${\max}_{ j\geq J_\epsilon} (\kappa_j j^{-\beta})  \|f_0\|_{\beta}\leq \epsilon$ satisfies that $J_\epsilon\asymp  (\|f_0\|_{\beta}/\epsilon)^{1/(\beta+p)}$, resulting in $\big\|{w^{J_\epsilon}}\big\|_{\mathbb{H}_{\mathcal{A}}}^2\lesssim \epsilon^{-\frac{2\alpha-2\beta+1}{\beta+p}}$ and proving the first statement. In the severely ill-posed case (with $\kappa_j\asymp e^{-c j^p},   \lambda_j\asymp j^{-\alpha}e^{-\xi j^p} $) the smallest $J_\epsilon\in\mathbb{N}$ such that $ {\max}_{ j\geq J_\epsilon} (\kappa_j j^{-\beta})  \|f_0\|_{\beta}\leq \epsilon$ implies that $\norm{w^{J_\epsilon}-w}_{L_2(\mathcal{X};G)}^2\lesssim J_\epsilon^{\alpha-2\beta}e^{\xi J_\epsilon^p}$
\end{proof}

\begin{lemma}[Small ball probability for random series priors]\label{lemma: small ball prob 0}
Consider the centered GP prior $\Pi_{\mathcal{A}}$  on $\mathcal{A}f$ given in \eqref{def:prior:inv}. Then there exists $C>0$ depending on $\alpha, p, c,\xi$ such that for any $\epsilon>0$ small enough
\[-\log \Pi_{\mathcal{A}} \left(w:\, \norm{w}_{L_2(\mathcal{X};G)}<\epsilon\right) \leq C\begin{cases} \epsilon^{-1/(\alpha+p)} &\mbox{if  $\kappa_j\asymp j^{-p}, \lambda_j\asymp j^{-1-2\alpha}$ for $\alpha>0,p\geq0$ }, \\
&\\
\log^{(p+1)/p} \frac{1}{\epsilon}& \mbox{if $\kappa_j\asymp e^{-c j^p},   \lambda_j\asymp j^{-\alpha}e^{-\xi j^p} $},\\
&\mbox{for $\alpha \geq 0, \xi>0\text{ or }\xi=0, \alpha\geq 2\beta, \text{ and $p\geq1$}$}.
\end{cases}\] 
\end{lemma}
\begin{proof}
The first case (polynomial decay) was derived in Lemma 11.47 from \cite{fundamentalsBNP}. In the second case, for $J\geq1$ and $Z_j\sim^{iid}N(0,1)$,
\[ \Pi_{\mathcal{A}}\big(w:\, \norm{w}_{L_2(\mathcal{X};G)}<\epsilon\big) \geq P\Big(\sum_{j\leq J} \lambda_j\kappa_j^{2} Z_j^2<\epsilon^2/2\Big)P\Big(\sum_{j>J}  \lambda_j\kappa_j^{2} Z_j^2<\epsilon^2/2\Big).\]
Note that the likelihood ratio of centered Gaussians with standard deviations $\sigma\geq\tau$ satisfy $\psi_\sigma/\psi_\tau(x) \geq \tau/\sigma$ uniformly on $x\in\mathbb{R}$. Therefore, the first term on the rhs of the preceding display is bounded from below by
\begin{align*}
&P\Big(\sum_{j\leq J} j^{-\alpha}e^{-(\xi+2c) j^p} Z_j^2<c\epsilon^2\Big)\geq \\
&e^{(c+\xi/2)\Big(\sum_{j=1}^J j^p -J^{p+1}\Big)} \prod_{j=1}^J \Big(\frac{j}{J}\Big)^{\alpha/2} P\Big(\sum_{j\leq J} J^{-\alpha}e^{-(\xi+2c) J^p} Z_j^2<c\epsilon^2\Big).
\end{align*}
The logarithm of the leading factor is equivalent to $-\frac{p}{p+1}(c+\xi/2)J^{p+1}$ as $J\to\infty$. The second is lower bounded by $(J!/J^J)^{\alpha/2}\geq e^{-J\alpha/2}$. By the central limit theorem, the probability in the last factor is greater than $1/2$ as $J\to\infty$ as long as $J^{\alpha-1}e^{(\xi+2c) J^p} \epsilon^2c \geq 2$. Also, by Markov's inequality,
\[ P\Big(\sum_{j>J} \lambda_j\kappa_j^2 Z_j^2<\epsilon^2/2\Big)\geq 1-2\epsilon^{-2}\sum_{j>J}E(Z_j^2\lambda_j\kappa_j^2)\geq 1-c_1\epsilon^{-2}\sum_{j>J}j^{-\alpha}e^{-(\xi+2c) j^p}.\]
Since the above sum is smaller than $c_2\epsilon^{-2} J^{-\alpha}e^{-(\xi+2c) J^p}$ (following form the assumption $p\geq 1$ and the sum of geometric series), the above probability is greater than $1/2$ whenever $J^{\alpha}e^{(\xi+2c) J^p}\geq 2c_2\epsilon^{-2}$. Therefore, as long as $J^{\alpha-1}e^{(\xi+2c) J^p} \epsilon^2 \geq (2/c)\vee (2c_2)$,
\[ -\log \Pi_{\mathcal{A}}\left(w: \norm{w}_{L_2(\mathcal{X};G)}<\epsilon\right) \lesssim J^{p+1}.\]
The above conditions are satisfied for $J\asymp \log^{1/p} \epsilon^{-1}$, concluding the proof of the lemma.
\end{proof}

\begin{lemma}[Theorem 5 of \cite{ray2022variational}]\label{lemma: var rates transition}
Let $C_n$ be a measurable subset of the parameter space $L_2\left(\mathcal{T}; \mu\right)$, $A_n$ be an event and $Q$ a distribution on $L_2\left(\mathcal{T}; \mu\right)$. If there exists $C>0$ and $\Delta_n\rightarrow\infty$ such that
\[ E_{f_0}\Pi\left[C_n^c\given X,Y\right]\mathds{1}_{A_n} \leq Ce^{-\Delta_n},\]
then
\[E_{f_0}Q\left(C_n^c\right)\mathds{1}_{A_n} \leq \frac{2}{\Delta_n}\left[E_{f_0} KL\big(Q|\! |\Pi[\cdot\given X ,Y]\big)+Ce^{-\Delta_n/2}\right].\]
\end{lemma}

\begin{lemma}\label{lemma: post mass events}
Let $\mathcal{S}_n \subset L_2(\mathcal{T};\mu)$ be a measurable event such that for some $\epsilon_n\to 0$, $n\epsilon_n^2\to \infty$, and $ C>1$ large enough,
\[ \frac{\Pi [\mathcal{S}_n]}{\Pi[f\colon\ \norm{\mathcal{A}f-\mathcal{A}f_0}_{L_2\left(\mathcal{X};G\right)}\leq \epsilon_n ] } \leq e^{-Cn\epsilon_n^2}.\]
Then there exists an event $A_n\subset\mathcal{X}^n$, with $P_{f_0}\left(A_n\right)\to1$, and $C'>C/2$ such that
\[ E_{f_0}\Pi\left[\mathcal{S}_n\given X,Y\right]\mathds{1}_{A_n}\lesssim e^{-C'n\epsilon_n^2}.\]
\end{lemma}
\begin{proof}
For $KL(f_0\| f)=P_{f_0}\log\left(dP_{f_0}/dP_f\right)$ and $V\left(f_0\| f\right)=P_{f_0}|\log\left(dP_{f_0}/dP_f\right)|^2$, in the random design regression model, in view of Lemma 2.7 of \cite{fundamentalsBNP}, the neighbourhood
\[ B_2(f_0; \epsilon_n)\coloneqq \left(f\colon\ KL(f_0\|f)\leq n\epsilon_n^2,\ V\left(f_0\| f\right)\leq n\epsilon_n^2\right)\]
contains the ball $\big\{f\colon  \norm{\mathcal{A}f-\mathcal{A}f_0}_{L_2\left(\mathcal{X};G\right)}\leq \epsilon_n \big\}$. Therefore, 
\[ \Pi \big[f\colon\ \norm{\mathcal{A}f-\mathcal{A}f_0}_{L_2\left(\mathcal{X};G\right)}\leq \epsilon_n \big] \leq \Pi\left[B_2(f_0; \epsilon_n)\right].\]
By Lemma 8.10 in \cite{fundamentalsBNP}, for any $c>1$, there exists an event $A_n^c$ of vanishing mass such that on $A_n$
\[ \int dP_{f}/dP_{f_0}(X,Y)\Pi(df) \geq \Pi\left[B_2(f_0; \epsilon_n)\right]e^{-cn\epsilon_n^2}. \]
Let us define  $B_n=A_n\cap \{\psi=0\}$ for any $\psi\colon \left(\mathcal{X}\times \mathbb{R}\right)^{n}\mapsto \{0,1\}$ such that $E_{f_0}\psi \to 0$, implying $P(B_n^c)=o(1)$. Then, taking $c<C$, 
\begin{align*}
E_{f_0}\Pi\left[\mathcal{S}_n\given X,Y\right]\mathds{1}_{B_n}&=  E_{f_0}\frac{\int_{\mathcal{S}_n} dP_{f}/dP_{f_0}(X,Y) \left(1-\psi\right)(X,Y)d\Pi (f)}{\int dP_{f}/dP_{f_0}(X,Y)d\Pi (f)}\mathds{1}_{B_n}\\
&\lesssim e^{cn\epsilon_n^2} \frac{\int_{\mathcal{S}_n} E_{f_0}dP_{f}/dP_{f_0} \left(1-\psi\right)d\Pi (f)}{\Pi[B_2(f_0; \epsilon_n)]}\\
&\lesssim  e^{cn\epsilon_n^2} \frac{\int_{\mathcal{S}_n} E_f \left(1-\psi\right)d\Pi (f)}{  \Pi[f\colon\ \norm{\mathcal{A}f-\mathcal{A}f_0}_{L_2(\mathcal{X};G)}\leq \epsilon_n ] }\\
&\lesssim e^{cn\epsilon_n^2}   \frac{\Pi[\mathcal{S}_n]}{\Pi[f\colon\ \norm{\mathcal{A}f-\mathcal{A}f_0}_{L_2(\mathcal{X};G)}\leq \epsilon_n]}
\lesssim e^{-C'n\epsilon_n^2}.
\end{align*}
\end{proof}

%

\begin{lemma}\label{lem:norms}
Assume that $\norm{g_j}_\infty \lesssim j^\gamma$ and that $\alpha+p>1+2\gamma$ in case of the mildly ill-posed inverse problem.
Then, there exists an event $B_n\subset \mathcal{X}^n$ with $P_X\big(B_n^c\big)=o(1)$, $C>0$ and a measurable subset $\mathcal{G}_n\subset L_2(\mathcal{X};G)$ with $\Pi(f:\, \mathcal{A}f\in \mathcal{G}_n^c)=o(e^{-n^{\frac{1}{1+2\gamma}}n\varepsilon_n^2})$ satisfying
\begin{align}
 \|w\|_{L_2(\mathcal{X};G)}^2\leq C (\|w\|_{L_2(\mathcal{X};P_n)}^2+\varepsilon_n^2)\label{eq:compare:norms}
\end{align}
and
\begin{align}
\|w\|_{L_2(\mathcal{X};P_n)}^2\leq C ( \|w\|_{L_2(\mathcal{X};G)}^2+\varepsilon_n^2)\label{eq:compare:norms2}
\end{align}
for any $X\in B_n$ and $w\in \mathcal{G}_n$.
\end{lemma}

\begin{proof}

Let us take $k=n^{\frac{1}{1+2\gamma}}/\log^{2/(1+2\gamma)} n$ and define the sieve
 \begin{equation}\label{def: sieve}
   \mathcal{F}_n=\{f\in  L_2(\mathcal{T};\mu):\, \mathcal{A}f\in\mathcal{G}_n\},\text{ with } \mathcal{G}_n=\{ w\in L_2(\mathcal{X};G):\,\ \|w^{\perp k} \|_\infty\leq  \varepsilon_n\},
 \end{equation} 
 where $w^{\perp k}(x)=\sum_{j=k+1}^\infty w_j g_j(x)$. Similarly we will denote by $w^{k}(x)=\sum_{j=1}^k w_j g_j(x)$. Next we show that $\Pi (\mathcal{F}_n^c)=o(e^{-n^{\frac{1}{1+2\gamma}}n\varepsilon_n^2})$.

First note that the assumption $\norm{g_j}_\infty \lesssim j^\gamma$ implies that $\|w^{\perp k}\|_{\infty}\leq C \sum_{j=k+1}^\infty  j^{\gamma}|w_j|$. Under the prior $\Pi$ on $f$, we have $f_j = \langle f, g_j\rangle_{L_2(\mathcal{X};G)} \stackrel{d}{=} \lambda_j^{1/2} Z_j$ with $Z_j\sim^{iid}N(0,1)$, therefore
 \begin{align*}
 \Pi(f:\, \|\mathcal{A}f^{\perp k} \|_\infty>  \varepsilon_n )\leq P\big( C \sum_{j=k+1}^\infty \kappa_j \lambda_j^{1/2} j^{\gamma}|Z_j|>  \varepsilon_n \big)=o(e^{-n^{\frac{1}{1+2\gamma}}n\varepsilon_n^2}),
 \end{align*}
 where the last equation follows from Lemma \ref{lemma: supnorm high indices prob} and $\alpha+p>1+2\gamma$. The above two displays together imply that $\Pi(f:\, \mathcal{A}f\in \mathcal{G}_n^c)\lesssim  e^{-C^on\varepsilon_n^2}$.

It remains to show that for $w\in\mathcal{G}_n$ there exists an event $B_n$ with $P_{X}\big( B_n\big)\rightarrow 1$, such that for $X\in B_n$ the inequalities \eqref{eq:compare:norms} and \eqref{eq:compare:norms2} hold.  This follows from the fact that for $w\in\mathcal{G}_n$,
\[\|w^{\perp k}\|_{L_2(\mathcal{X};G)}\ \vee\ \|w^{\perp k}\|_{L_2(\mathcal{X};P_n)}\leq \|w^{\perp k}\|_\infty\leq\varepsilon_n.\]
This inequality also allows to write, under the event of Lemma \ref{lem:rudelson} which we note $B_n$, that
\begin{align*}
 \|w\|_{L_2(\mathcal{X};P_n)}\leq \|w^k\|_{L_2(\mathcal{X};P_n)}+ \|w^{\perp k}\|_{L_2(\mathcal{X};P_n)}\lesssim
  \|w^k\|_{L_2(\mathcal{X};G)}+\varepsilon_n\leq \|w\|_{L_2(\mathcal{X};G)}+\varepsilon_n,
\end{align*}
proving \eqref{eq:compare:norms} (a similar argument proves \eqref{eq:compare:norms2}).

\end{proof}

\begin{lemma}\label{lem:rudelson}
For $k=n^{\frac{1}{1+2\gamma}}/\log^{2/(1+2\gamma)} n$,  $\gamma\geq 0$, there exists a constant $C_0>1$ such that, with $P_{X}$-probability tending to one,
\[C_0^{-1}\|w^k\|_{ L_2(\mathcal{X};G)}\leq  \|w^k\|_ {L_2(\mathcal{X};P_n)}\leq C_0\|w^k\|_{ L_2(\mathcal{X};G)},\]
for any $w\in L_2(\mathcal{X};G)$, and where $w^k(x)=\sum_{j=1}^k w_j g_j(x)$ is the orthogonal projection on the $k$ first elements of an orthonormal basis $(g_j)_{j\in\mathbb{N}}$ satisfying $\norm{g_j}_\infty\lesssim j^\gamma$.
\end{lemma}

\begin{proof}
First we introduce $\Sigma_{n,k}=n^{-1}\boldsymbol{G}_{n,k}^T\boldsymbol{G}_{n,k}$ with $\boldsymbol{G}_{n,k}=(\boldsymbol{g} (X_1),....,\boldsymbol{g} (X_n))^T\in \mathbb{R}^{n\times k}$, with $\boldsymbol{g} (X_1)=\big(g_1(X_1),...,g_k(X_1)\big)^T$. Note that $E_{\mathbb{X}}\Sigma_{n,k}=I_k$ as the eigenbasis $(g_j)_{j\in\mathbb{N}}$ is orthonormal w.r.t. the design distribution $G$. Then by the modified version of Rudelson's inequality \cite{rudelson} we get that
\begin{align*}
E_{\mathbb{X}}\| \Sigma_{n,k}-I_k \|_2\leq C \sqrt{\frac{\log k}{n}}E_{\mathbb{X}}(\|\boldsymbol{g} (X_1) \|_2^{\log n})^{1/\log n}.
\end{align*}
Note that by the boundedness assumption $\sum_{j=1}^k g_j(x)^2\leq C k^{1+2\gamma}$, $x\in\mathcal{X}$, so that the right hand side of the preceding display is bounded from above by constant times $\sqrt{k^{1+2\gamma}\log (k)/n}=o(1)$. Therefore, noting $\boldsymbol{w}=(w_1,...,w_k)$ for $w\in L_2(\mathcal{X};G)$, 
\begin{align*}
\underset{w\in L_2(\mathcal{X};G)}{\sup} \frac{\Big|\|w^k\|_{ L_2(\mathcal{X};G)}^2-\|w^k\|_{{ L_2(\mathcal{X};P_n)}}^2\Big|}{ \|w^k\|_{ L_2(\mathcal{X};G)}^2}&= \underset{w\in L_2(\mathcal{X};G)}{\sup} \Big| \boldsymbol{w}^{T} (I_k-\Sigma_{n,k})\boldsymbol{w}\Big|/ \|w^k\|_{ L_2(\mathcal{X};G)}^2\\
&\leq \underset{w\in L_2(\mathcal{X};G)}{\sup}  \|I_k-\Sigma_{n,k}\|_2 \|\boldsymbol{w}\|_2^2/ \|w^k\|_{ L_2(\mathcal{X};G)}^2\\
&=o_{P_{\mathbb{X}}}(1). 
\end{align*}
Then, on an event $A_n(\mathbb{X})$ with $P_{\mathbb{X}}(A_n(\mathbb{X}))$ tending to one, for all $w\in\mathcal{G}_n$
\begin{align}
 \|w^k\|_{ L_2(\mathcal{X};G)}^2/2 \leq \|w^k\|_{{ L_2(\mathcal{X};P_n)}}^2\leq 2 \|w^k\|_{ L_2(\mathcal{X};G)}^2,\label{eq:reg:metrics}
\end{align}
for any $w$, verifying the statement.
\end{proof}

\begin{lemma}\label{lemma: supnorm high indices prob}
Assume that $ \nu_j\leq  C j^{-3/2-\delta}$, $\delta>\gamma\geq 0$ and that $n\varepsilon_n^2\rightarrow\infty$. Then, for $Z_j$ independent standard normal random variables and any $C'>0$,
\begin{align*}
P\left(\sum_{j=n^{\frac{1}{1+2\gamma}}/\log^{2/(1+2\gamma)} n}^{\infty}\nu_j |Z_j|\geq C' \varepsilon_n\right)=o(e^{-n^{\frac{1}{1+2\gamma}}n\varepsilon_n^2}).
\end{align*}
\end{lemma}

\begin{proof}
Let us introduce the notation $k=n^{\frac{1}{1+2\gamma}}/\log^{2/(1+2\gamma)} n$ and note that, for any $C_1>0$, there exist positive constants $C_2,C_3$ such that
\begin{align}
P\Big(\sum_{j=k}^{\infty}\nu_j |Z_j|\geq C_1 \varepsilon_n\Big)&\leq \sum_{i=1}^{\infty} P\Big(\sum_{j=ik}^{(i+1)k-1}\nu_j|Z_j|\geq C_2(i\left(1+\log^2 i)\right)^{-1} \varepsilon_n \Big)\nonumber\\
&\leq   \sum_{i=1}^{\infty} P\Big(\sum_{j=ik}^{(i+1)k-1} C i^{-3/2-\delta}k^{-3/2-\delta} |Z_j|\geq  C_2\left(i(1+\log^2 i)\right)^{-1}\varepsilon_n\Big)\nonumber\\
&\leq \sum_{i=1}^{\infty} P\Big(\sum_{j=ik}^{(i+1)k-1}  |Z_j|\geq  C_3 k^{3/2+\delta}i^{1/2}\varepsilon_n\Big)
.\label{eq:help:sum}
\end{align}
We show below that 
\begin{align}
 P\left(\sum_{j=ik}^{(i+1)k-1} |Z_j|\geq  C_3k^{3/2+\delta}i^{1/2}\varepsilon_n\right)\leq 2^{k} e^{-c ik^{2+2\delta}\varepsilon_n^2}. \label{eq:help:Bernstein}
\end{align}
which in turn implies (together with $n\varepsilon_n^2 \rightarrow\infty$ and $\delta>\gamma\geq0$) that the rhs of \eqref{eq:help:sum} is further bounded by
\begin{align*}
\sum_{i=1}^{\infty} 2^{k} e^{-c_1ik^{2+2\delta}\varepsilon_n^2} \lesssim 2^k e^{- c_1 k^{2+2\delta}\varepsilon_n^2}=o(e^{-n^{\frac{1}{1+2\gamma}}n\varepsilon_n^2}).
\end{align*}
It remained to prove \eqref{eq:help:Bernstein}. For convenience, let us introduce the notation $c_{i,k}=C_3i^{1/2} k^{3/2+\delta} $. Following the proof of Chernoff's inequality and recalling that the characteristic function of the absolute value of the standard normal distribution satisfies that $Ee^{t|Z|}\leq 2 e^{t^2/2}$, we get for $\gamma= c_{i,k} \varepsilon_n/k$ that
\begin{align*}
 P\left(\sum_{j=ik}^{(i+1)k-1} |Z_j|\geq c_{i,k} \varepsilon_n\right)&=  P\left(e^{\gamma\sum_{j=ik}^{(i+1)k-1}|Z_j|}\geq e^{\gamma c_{i,k} \varepsilon_n}\right)\\
 &\leq e^{-\gamma c_{i,k} \varepsilon_n}E e^{\gamma  \sum_{j=ik}^{(i+1)k-1} |Z_j|}\leq e^{-\gamma c_{i,k} \varepsilon_n} \prod_{j=ik}^{(i+1)k-1} 2e^{\gamma^2/2}\\
 &= 2^k e^{k\gamma^2/2-\gamma c_{i,k} \varepsilon_n}=2^k  e^{- c_{i,k}^2 \varepsilon_n^2/(2k)}.
\end{align*}
\end{proof}

\begin{lemma}\label{lem:Af_0}
Let $f_0\in \bar{H}^{\beta}$, for some $\beta>0$, $J\in\mathbb{N}$ and assume that $\|g_j\|_{\infty}\leq Cj^{\gamma}$ for some $\gamma\geq0$. Furthermore, in case of the mildly ill-posed inverse problem assume that $p+\beta-\gamma>1$. Then $P_X$-almost surely
\begin{align*}
\|\mathcal{A}f_0^{\perp J}\|_{L^{2}(\mathcal{X},P_n)}\lesssim   q(J),
\end{align*}
where $q(J)=J^{-p-\beta+\gamma+1}$ in the mildly and $q(J)=J^{\gamma-\beta} e^{-cJ^p}$ in the severely ill-posed inverse problem. Furthermore, for any $J\leq k=n^{\frac{1}{1+2\gamma}}/\log^{2/(1+2\gamma)} n$ , with $P_X$-probability tending to one
\begin{align*}
\|\mathcal{A}f_0^{\perp J}\|_{L^{2}(\mathcal{X},P_n)}\lesssim \|\mathcal{A}f_0^{\perp J}\|_{L^{2}(\mathcal{X},G)}+q\left(n^{\frac{1}{1+2\gamma}}/\log^{2/(1+2\gamma)} n\right).
\end{align*}
\end{lemma}

\begin{proof}
We start with the first assertion. In view of $|f_{0,j}|\leq j^{-\beta} \|f_0\|_\beta$ and triangle inequality one can observe that
\begin{align*}
\|\mathcal{A} f_0^{\perp J}\|_{L^{2}(\mathcal{X},P_n)}\leq \|\mathcal{A}f_0^{\perp J}\|_\infty\leq \sum_{j= J+1}^{\infty} \kappa_j|f_{0,j}|\|g_j\|_\infty \lesssim \sum_{j= J+1}^{\infty} \kappa_j j^{\gamma-\beta}.
\end{align*}
Then for the mildly ill-posed inverse problem (with $\kappa_j\asymp j^{-p}$) the rhs of the preceding display is further bounded by a multiple of $J^{-p-\beta+\gamma+1}$, while in the severely ill-posed inverse problem (with $\kappa_j\asymp e^{-c j^p}$) it is bounded from above by a multiple of $J^{\gamma-\beta} e^{-cJ^p}$ since $p\geq1$.

For the second assertion of the lemma, note that for $J\leq k=n^{\frac{1}{1+2\gamma}}/\log^{2/(1+2\gamma)} n$, by triangle inequality
\begin{align*}
\|\mathcal{A} f_0^{\perp J}\|_{L^{2}(\mathcal{X},P_n)}\leq \|\mathcal{A} f_0^{\perp J}-\mathcal{A} f_0^{\perp k}\|_{L^{2}(\mathcal{X},P_n)}+
\|\mathcal{A} f_0^{\perp k}\|_{L^{2}(\mathcal{X},P_n)}.
\end{align*}
The first term, in view of  Lemma \ref{lem:rudelson}, is bounded by a multiple of $\|\mathcal{A} f_0^{\perp J}-\mathcal{A} f_0^{\perp k}\|_{L^{2}(\mathcal{X},G)}\leq \|\mathcal{A} f_0^{\perp J}\|_{L^{2}(\mathcal{X},G)} $ with $P_X$-probability tending to one by, while the second term is bounded by a multiple of $q\left(n^{\frac{1}{1+2\gamma}}/\log^{2/(1+2\gamma)} n\right)$ following from the first statement of the lemma.
\end{proof}

\section{Proof of Corollary \ref{lemma: number of inducing variables}}

For the first choice \eqref{eq: inducing var matrix} in the mildy ill-posed case, Lemma 4 of \cite{JMLR:v23:21-1128} combined with the polynomial decay of the eigenvalues $\lambda_i\kappa_i^2$ of the process with kernel \eqref{def:prior:inv} gives that
\begin{align*}
E_X \norm{K_{\boldsymbol{\mathcal{A}f}\boldsymbol{\mathcal{A}f}}-Q_{\boldsymbol{\mathcal{A}f}\boldsymbol{\mathcal{A}f}}}&\lesssim nm^{-1-2(\alpha+p)},\\
E_X Tr\left(K_{\boldsymbol{\mathcal{A}f}\boldsymbol{\mathcal{A}f}}-Q_{\boldsymbol{\mathcal{A}f}\boldsymbol{\mathcal{A}f}}\right) &\lesssim  nm^{-2(\alpha+p)}.
\end{align*}
In view of Theorem \ref{th: post rates}, we set\[m=m_n=n^{\frac{1}{1+2p+2\alpha}}\] to translate posterior contraction rates into variational ones.
For the second case \eqref{eq: inducing var op}, since $\alpha+p>1$ and $\underset{j}{\sup}\ \underset{x}{\sup}\ |g_j(x)|<\infty$ under our assumptions, the bounds come from Lemma 5 of \cite{JMLR:v23:21-1128} and are 
\begin{align*}
E_X \norm{K_{\boldsymbol{\mathcal{A}f}\boldsymbol{\mathcal{A}f}}-Q_{\boldsymbol{\mathcal{A}f}\boldsymbol{\mathcal{A}f}}}&\lesssim 1 + nm^{-1-2(\alpha+p)} + n^{\frac{1}{2(\alpha+p)}}m^{-2(\alpha+p)}\log n,\\
E_X Tr\left(K_{\boldsymbol{\mathcal{A}f}\boldsymbol{\mathcal{A}f}}-Q_{\boldsymbol{\mathcal{A}f}\boldsymbol{\mathcal{A}f}}\right) &\lesssim  nm^{-2(\alpha+p)}.
\end{align*}Then, $m=m_n$ as above is sufficient as well.

Finishing with the severely ill-posed problem, we have for both choices of inducing variables
\[ E_X \norm{K_{\boldsymbol{\mathcal{A}f}\boldsymbol{\mathcal{A}f}}-Q_{\boldsymbol{\mathcal{A}f}\boldsymbol{\mathcal{A}f}}}\leq E_X Tr\left(K_{\boldsymbol{\mathcal{A}f}\boldsymbol{\mathcal{A}f}}-Q_{\boldsymbol{\mathcal{A}f}\boldsymbol{\mathcal{A}f}}\right) \leq n\sum_{j>m}\lambda_j\lesssim nm^{-\alpha}e^{-(\xi+2c)m^p},\]where the second inequality comes from Proposition 2 in \cite{ShaweTaylor2002TheSO}. Then, $m=m_n = \left((\xi+2c)\log n\right)^{1/p}$ is sufficient.\\

\textit{Funding.}  Co-funded by the European Union (ERC, BigBayesUQ, project number: 101041064). Views and opinions expressed are however those of the author(s) only and do not necessarily reflect those of the European Union or the European Research Council. Neither the European Union nor the granting authority can be held responsible for them.

{
\small

\bibliographystyle{acm}
 \bibliography{biblio}
}

 \end{document}